\documentclass[11pt]{article}

\usepackage{amsmath,amssymb,latexsym,amsfonts,mathrsfs,graphicx,verbatim}
\usepackage{algpseudocode,algorithm,algorithmicx}

\usepackage{graphicx}
\usepackage{setspace}
\usepackage{program}
\usepackage{subfig}

\renewcommand{\quad}{$~~~\;\;\;$}
\newtheorem{theorem}{Theorem}

\newtheorem{proposition}{Proposition}

\newtheorem{example}{Example}

\newcommand{\pyrwidth}{\text{\rm PWidth}}
\newcommand{\pyrdir}{\text{\rm PdirW}}
\newcommand{\faces}{\text{\rm faces}}
\newcommand{\facets}{\text{\rm facets}}
\newcommand{\vertices}{\text{\rm vertices}}

\newcommand{\conv}{\text{\rm conv}}
\newcommand{\affine}{\text{\rm affine}}
\newcommand{\transp}{{^{\rm T}}}

\newcommand{\cone}{{\rm cone\,}}
\newcommand{\vertiii}[1]{{\left\vert\kern-0.25ex\left\vert\kern-0.25ex\left\vert #1 
    \right\vert\kern-0.25ex\right\vert\kern-0.25ex\right\vert}}

\newcommand{\diam}{\mathsf{diam}}

\renewcommand{\R}{\mathbb{R}}

\newcommand{\matr}[1]{\begin{bmatrix} #1 \end{bmatrix}}    
\def\transp{^{\rm T}}

\newcommand{\ip}[2]{\left\langle #1 , #2 \right\rangle}    

\providecommand{\newoperator}[3]{%
  \newcommand*{#1}{\mathop{#2}#3}}

\newoperator{\argmax}{\mathrm{argmax}}{}
\newoperator{\argmin}{\mathrm{argmin}}{}
\newoperator{\Argmin}{\mathrm{Argmin}}{}
\newoperator{\Argmax}{\mathrm{Argmax}}{}
\newcommand{\1}{\mathbf{1}}
\newcommand{\dmin}{\displaystyle\min}
\newcommand{\dmax}{\displaystyle\max}
\newcommand{\dist}{\mathsf{dist}}

\usepackage{color}

\newcommand{\blue}[1]{#1}

\author{Javier Pe\~na\thanks{Tepper School of Business,
Carnegie Mellon University, USA, {\tt jfp@andrew.cmu.edu}}
\and Daniel Rodr\'iguez\thanks{Department of Mathematical Sciences, Carnegie
Mellon University, USA, {\tt drod@cmu.edu }}
}
\title{Polytope conditioning and linear convergence of the Frank-Wolfe algorithm}

\begin{document}
\maketitle
\begin{abstract}
It is known that the gradient descent algorithm converges linearly when applied to a strongly convex function with Lipschitz gradient. In this case the algorithm's rate of convergence is determined by the condition number of the function.  In a similar vein, it has been shown that a variant of the Frank-Wolfe algorithm with away steps converges linearly when applied to a strongly convex function with Lipschitz gradient over a polytope.  In a nice extension of the unconstrained case, the algorithm's rate of convergence is determined by the product of the condition number of the function and a certain condition number of the polytope.  

We shed new light into the latter type of polytope conditioning.  In particular, we show that previous and seemingly different approaches to define a suitable condition measure for the polytope are essentially equivalent to each other.  Perhaps more interesting, they can all be unified via a parameter of the polytope that formalizes a key premise linked to the algorithm's linear convergence.  We also give new insight into the linear convergence property. For a convex quadratic objective, we show that the  rate of convergence is determined by \blue{a} condition number of a suitably scaled polytope.
\end{abstract}

\newpage

\section{Introduction}

It is a standard result in convex optimization that the gradient descent algorithm 
converges linearly to the minimizer of a strongly convex function with Lipschitz gradient.  For a related discussion, e.g.,~\cite[Chapter 2]{Nest04} or~\cite[Chapter 1]{Bert99}.  Furthermore, in this case the rate of convergence is determined by the {\em condition number} of the objective function, that is, the ratio between the Lipschitz parameter of the gradient and the strong convexity parameter of the function.

In analogous fashion, the Frank-Wolfe algorithm~\cite{FranW56,Jagg13}, also known as conditional gradient algorithm, for the problem
$
\dmin_{\blue{u}\in C} f(\blue{u})
$
converges linearly to the minimizer of $f$ on a compact convex set $C$ provided $f$ is strongly convex \blue{with Lipschitz gradient} and the optimal solution lies in the relative interior of $C$.  For a related discussion \blue{see, e.g.,~\cite{BeckT04,EpelF00,GuelM86,LacoJ15}} and the references therein. The assumption that the optimal solution belongs to relative interior of $C$ is critical for the linear convergence of the algorithm.  Indeed, the rate of convergence depends on how far the optimal solution is from the relative boundary of $C$. In particular, this rate deteriorates when the optimal solution is near the relative boundary of $C$, and linear convergence breaks down altogether when the optimal solution is on the relative boundary of $C$.  

\blue{A variant of the Frank-Wolfe algorithm that includes {\em away steps} was proposed by Wolfe~\cite{Wolf70}.  Several articles have shown linear convergence results for the away step variant and for other variants of the Frank-Wolfe algorithm when the domain $C$ is a polyhedron.  The article~\cite{GuelM86} establishes a local linear convergence result for the away step variant for strongly convex with Lipschitz gradient under a certain kind of strict complementarity assumption.   The articles~\cite{AhipST08,KumaY11} give local linear convergence results for smooth convex functions by relying on Robinson's second-order constraint qualification.   
On the other hand, \cite{NancFSA14} shows linear convergence results for a pairwise variant of the Frank-Wolfe algorithm.  The article \cite{GarbH13} shows linear convergence for a version of the Frank-Wolfe algorithm that relies on a local linear optimization oracle. } 

\blue{
The recent articles~\cite{BeckS15,LacoJ15,PenaRS15} establish {\em global} linear convergence results for the Frank-Wolfe algorithm with away steps when the objective function is strongly convex  with Lipschitz gradient and the domain is of the form $C = \conv(A)$ for a finite set of atoms $A$.  It should be noted that both \cite{BeckS15} and~\cite{PenaRS15} were inspired by and relied upon key ideas and results first introduced in a preliminary workshop version of~\cite{LacoJ15}.
 A common feature of~\cite{BeckS15,LacoJ15,PenaRS15} is that the statement of linear convergence is given in terms of the condition number of the objective function $f$ and some type of {\em condition number} of the polytope $\conv(A)$.}  As we detail in Section~\ref{sec.FrankWolfe}, a generic version of linear convergence as in~\cite{BeckS15,LacoJ15,PenaRS15} hinges on three main premises.  The first premise is a certain {\em slope bound} on the objective function and its optimal solution set.  This first premise readily holds for strongly convex functions as it does in the unconstrained case.  
The second premise is a {\em decrease condition} on the objective function at each iteration of the algorithm.  As in the unconstrained case, the second premise holds as long as an upper bound on the Lipschitz constant of the gradient is available, or if an appropriate line-search is performed at each iteration.  The third premise, which seems to be the main technical component in each of the papers~\cite{BeckS15,LacoJ15,PenaRS15}, is a premise on the search direction selected by the algorithm at each iteration.  
Loosely speaking, this third premise is a condition on the alignment of the search direction with the gradient of the objective function at the current iterate.  The premise is that this alignment should be comparable to that of a direct step towards the optimal solution.  Unlike the first two premises, that essentially match the premises leading to the linear convergence of gradient descent in the unconstrained case, the third premise is inherent to the polytope defining the constraint set of the problem.   This third premise can be formalized in terms of a certain kind of {\em condition number} of the polytope.  In a nice extension of the unconstrained case, the rate of linear convergence of the Frank-Wolfe algorithm with away steps is determined by the product of the usual condition number of the objective function and the condition number of the polytope.  (See Theorem~\ref{thm.fw.linear} in Section~\ref{sec.FrankWolfe}.)

The central goal of this paper is to shed new light into this polytope condition number.  The three articles~\cite{BeckS15,LacoJ15,PenaRS15} make different attempts to define a suitable  condition measure along the lines of the third premise sketched above.  Each of these attempts has different merits and limitations.  One of this paper's main contributions is to show that these three kinds of condition measures, namely the {\em pyramidal width} defined by Lacoste-Julien and Jaggi~\cite{LacoJ15}, the {\em vertex-facet distance} defined by Beck and Shtern~\cite{BeckS15}, and the {\em restricted width} defined by Pe\~na, Rodr\'iguez, and Soheili~\cite{PenaRS15},  turn out to be essentially equivalent.  Perhaps more important, they are all unified via a {\em facial distance} of the polytope.  As we explain in Section~\ref{sec.main} and Section~\ref{sec.FrankWolfe}, the facial distance  can be seen as a natural quantity associated to the polytope that formalizes a key alignment condition of the search direction at each iteration of the Frank-Wolfe algorithm with away steps.  

Section~\ref{sec.main} presents the technical bulk of our paper.  
One of our  results (Theorem~\ref{main.thm}) is a characterization of the facial distance  of a polytope as the minimum distance between a \blue{proper} face of the polytope and a kind of {\em complement polytope}.  This characterization can be seen as a refinement of the {\em vertex-facet distance} proposed by Beck and Shtern~\cite{BeckS15}.  
Theorem~\ref{main.thm} motivates the name ``facial distance'' for this quantity.  Theorem~\ref{main.thm} provides a method to compute or bound the facial distance  as we illustrate in a few examples.   We also show (Theorem~\ref{main.thm.2}) that the facial distance  coincides with the pyramidal width defined by Lacoste-Julien and Jaggi~\cite{LacoJ15}.  As a byproduct of this result, we obtain a simplification of the original definition of pyramidal width.  We also give a  {\em localized} version of Theorem~\ref{main.thm} for a kind of {\em localized} version of the facial distance of the polytope (Theorem~\ref{prop.lower.bound}).

As mentioned above, Section~\ref{sec.FrankWolfe} details how the linear convergence of the Frank-Wolfe algorithm with away steps can be derived from three central premises.  The  goal of Section~\ref{sec.FrankWolfe} is to highlight the role of these three key premises, particularly the third one.  We discuss how the third premise is naturally tied to the facial distance of the polytope discussed in Section~\ref{sec.main}.  Our exposition  allows us to distill a key tradeoff in the existing bounds on the rate of convergence of the Frank-Wolfe algorithm with away steps.  On the one hand, the algorithm's rate of convergence can be bounded in terms of quantities that depend {\em only} on properties of the polytope and of the objective function but not on the optimal solution.  More precisely, for a strongly convex objective function with Lipschitz gradient the rate of convergence can be bounded in terms of the product of the condition number of the polytope and the condition number of the objective function.  This is a feature of the results \blue{in~\cite{BeckS15,LacoJ15} but not of those in~\cite{PenaRS15}} that depend on the optimal solution set. On the other hand, a {\em sharper} bound on the rate of convergence can be given if we allow it to depend on the location of the optimal solution in the polytope.   More precisely, the rate of convergence can be bounded in terms of the product of a {\em localized} condition number of the polytope  that depends on the solution set and the condition number of the objective function.  The statement of Theorem~\ref{thm.fw.linear} makes the connection between the two bounds completely transparent: The solution-independent bound is simply the most conservative solution-dependent one. Not surprisingly, the solution-dependent bound can be arbitrarily better than the solution-independent one.

Section~\ref{sec.quadratic} discusses the linear convergence property in the special but important case when the objective function is of the form $f(u) = \frac{1}{2}\ip{u}{Qu} + \ip{b}{u}$ for $Q$ positive semidefinite.  As Theorem~\ref{thm.quad.conv} in Section~\ref{sec.quadratic} shows, in this case the rate of convergence is determined by \blue{a variant of the facial distance of a suitably scaled polytope.}
  In Section~\ref{sec.non-convex} we show that the latter result extends, under suitable assumptions on the algorithm's choice of steplength, to a composite objective function of the form $f(u) = \blue{h}(Eu) + \ip{b}{u}$ where $\blue{h}$ is a strongly convex function with Lipschitz gradient and $E$ is a matrix of suitable size. (See Theorem~\ref{thm.non-strong} in Section~\ref{sec.non-convex}.)  This result is along the lines of the linear convergence result of Beck and Shtern's~\cite[Theorem 3.1]{BeckS15}.  However, our bound on the rate of convergence and proof technique are fairly different.  \blue{Both of them are  extensions of the three-premise approach described in Section~\ref{sec.FrankWolfe}.  We conclude our paper with some examples in Section~\ref{sec.examples} that illustrate the tightness of the linear convergence results stated in Theorem~\ref{thm.fw.linear} and Theorem~\ref{thm.quad.conv}.}

\blue
{
Throughout the paper we will often need to deal with multiple points in $\R^m$ and in $\R^n$.  We will consistently use $u,v$ to denote points in $\R^m$ and $w,x,y,z$ to denote points in $\R^n$.
}

\section{The facial distance  $\Phi(A)$}
\label{sec.main}
This section constitutes the technical bulk of the paper. We define the {\em facial distance  $\Phi(A)$} and prove several interesting results about it.  In particular, we show that it essentially matches the various kinds of condition measures previously defined in~\cite{BeckS15,LacoJ15,PenaRS15}.

Assume \blue{$A = \matr{a_1 & \cdots & a_n} \in \R^{m\times n}$.} For convenience we will make the following slight abuse of notation: We will write $A$ to denote both \blue{the matrix $\matr{a_1 & \cdots & a_n} \in \R^{m\times n}$ and the set of its columns $\{a_1,\dots,a_n\} \subseteq \R^m$.}  The appropriate meaning of $A$ will be clear from the context. Let $\Delta_{n-1}:=\{x\in\R^n_+: \|x\|_1 = 1\}.$   For $x\in \Delta_{n-1}$, define $I(x) \subseteq \{1,\dots,n\}$ and $S(x) \subseteq A$  as 
\[I(x):=\{i\in \{1,\dots,n\}: x_i > 0\}\]
and
$$S(x):=\{a_i: i \in I(x)\}.$$
Observe that the sets $I(x), S(x)$ define the {\em support} of $Ax$.  

Throughout the paper, $\|\cdot\|$ will denote the Euclidean norm.  
Assume $x,z\in \Delta_{n-1}$ are such that $A(x-z) \ne 0$ and let $d:=\frac{A(x-z)}{\|A(x-z)\|}$.  Define 
\begin{equation}\label{eq.def.phi}
\Phi(A,x,z) = 
\dmin_{p \in \R^m: \ip{p}{d}=1} 
\; \dmax_{s\in S(x), \, a\in A} \ip{p}{s-a},
\end{equation}
and
\[\Phi(A) = \min_{x,z\in \Delta_{n-1}:A(x-z) \ne 0} \Phi(A,x,z).
\]

\medskip

The connection between $\Phi(A,x,z)$ and the Frank-Wolfe algorithm with away steps algorithm will be made explicit as Premise 3 in Section~\ref{sec.FrankWolfe} but the basic idea is as follows.  At each iteration the algorithm starts from a current point $\blue{u}=Ax\in\conv(A)$ and selects the two atoms $a,s\in A$ that attain $\max_{s\in S(x), \, a\in A} \ip{p}{s-a}$ for $p =\nabla f(\blue{u})$.   Premise 3 requires that for $d := \frac{A(x-z)}{\|A(x-z)\|}$ the ratio $\frac{\ip{p}{s-a}}{\ip{p}{d}}$ be bounded away from zero, \blue{where $z\in \Delta_{n-1}$ is such that $u^* = Az$ is a solution to the minimization problem}.  The latter condition means the alignment of the vector $a-s$ and the direction $p$ should be comparable to the alignment of $d$ and $p.$  The need to formalize  this premise motivates the definition of the quantities $\Phi(A,x,z)$ and $\Phi(A)$.

\medskip

Notice the asymmetry between the roles $x$ and $z$ in $\Phi(A,x,z)$.  We can think of $z$ as defining an {\em anchor} point $Az\in \conv(A)$.  \blue{This anchor point determines a set of directions $d = \frac{A(x-z)}{\|A(x-z)\|}$ for $x\in\Delta_{n-1}$ with $A(x-z)\ne 0.$}  
 When $Az = 0$, the  quantity $\Phi(A,x,z)$ coincides with the quantity $\phi(A,x)$ defined in \cite{PenaRS15}.  Thus $\Phi(A)$ can be seen as a refinement of the {\em restricted width} $\phi(A) = \dmin_{x\in \Delta_{n-1}:Ax\ne 0} \phi(A,x)$ defined in~\cite{PenaRS15}.  \blue{When $0 \in \conv(A)$, we have $\Phi(A) \le \phi(A).$  More precisely, when $0 \in \conv(A)$ the restricted width $\phi(A)$ coincides with the {\em localized} variant 
 $\Phi(A,Z)$ of $\Phi(A)$ defined below for $Z = \{z\in \Delta_{n-1}:Az = 0\}$.} 
 
The quantity $\phi(A,x)$ was introduced in a form more closely related to the alternative expression \eqref{prop.identity} for $\Phi(A,x,z)$  in Proposition~\ref{prop.restricted.width} below.  The expression  \eqref{prop.identity} characterizes $\Phi(A,x,z)$ as the length of the longest segment in $\conv(A)$ in the direction $A(x-z)$ with one endpoint in $\conv(S(x))$ and the other in $\conv(A)$. 

\begin{proposition}\label{prop.restricted.width} Assume \blue{$A = \matr{a_1&\cdots&a_n} \in \R^{m\times n}$ and} $x,z\in \Delta_{n-1}$ are such that $A(x-z) \ne 0$ and  let $d:=\frac{A(x-z)}{\|A(x-z)\|}$. Then
\begin{equation}\label{prop.identity}
\Phi(A,x,z)= 
\blue{\dmax}\left\{\lambda > 0: \exists w,y\in \Delta_{n-1}, \; I(w) \subseteq I(x),\; A(w - y) = \lambda d\right\}.
\end{equation}
Furthermore, if $p$ attains the minimum value $\Phi(A,x,z)$ in \eqref{eq.def.phi} and $u=Aw, v = Ay$ maximize the right hand side in \eqref{prop.identity} then \blue {$v \in \conv(B)$ where $B = \displaystyle\Argmin_{a\in A} \ip{p}{a}$,
  $u \in \conv(A\setminus B)$, and $$\Phi(A,x,z) = \|u-v\|.$$}
\end{proposition}
\begin{proof}
To ease notation, let $I:=I(x)$.  Observe that the right-hand-side in \eqref{prop.identity} is
\[
\begin{array}{rl}
\dmax_{w_I,y,\lambda} & \lambda \\
& A_I w_I - Ay - \lambda d = 0 \\
& \1_I \transp w_I = 1 \\
& \1\transp y = 1 \\
& w_I, y \ge 0.
\end{array}
\]
On the other hand, from the definition~\eqref{eq.def.phi} of $\Phi(A,x,z)$, it follows that
\[
\begin{array}{rl}
\Phi(A,x,z) =  \dmin_{p,t,\tau} & t+\tau \\
& A_I\transp p \le  t \1_I  \\ 
& A\transp p \ge -\tau \1  \\
& \ip{d}{p} = 1. 
\end{array}
\]
Therefore \eqref{prop.identity} follows by linear programming duality. 

\blue{Next assume $p$ attains the minimum value $\lambda=\Phi(A,x,z)$ in \eqref{eq.def.phi} and $u=Aw, v = Ay$ maximize the right hand side in \eqref{prop.identity}.  Then $(w_I,y,\lambda)$ and $(p,t,\tau)$ are respectively solutions to the above pair of linear programs for $t = \max_{a\in A} \ip{p}{a}$ and $-\tau=\min_{a\in A} \ip{p}{a}$. By complementary slackness it follows that $y_i > 0$ only if $a_i\transp p = -\tau = \min_{a\in A} \ip{p}{a}.$ Thus $v = Ay \in \conv(B)$ for $B = \Argmin_{a\in A} \ip{p}{a}.$  Similarly, by complementary slackness it follows that $w_j > 0$ only if 
$a_j\transp p = t = \max_{a\in S(x)} \ip{p}{a}$.  Thus $u = Aw \in \conv(C)$ for $C=\Argmax_{a\in S(x)} \ip{p}{a}$. Next observe that $C \subseteq A\setminus B$ because $ \max_{a\in S(x)} \ip{p}{a} - \min_{a\in A} \ip{p}{a} = t+\tau = \lambda \ge \|A(x-z)\| > 0.$ 
Finally observe that 
\[
\|u-v\| = \|A(w-y)\| = \lambda \|d\| = \lambda = \Phi(A,x,z).
\]
}
\qed
\end{proof}

\bigskip

Theorem~\ref{main.thm} below gives a characterization of $\Phi(A)$ in terms of the minimum distance between a \blue{proper} face $F$ of $\conv(A)$ and its {\em complement polytope} $\conv(A\setminus F)$.  This characterization motivates the name {\em facial distance} for the quantity $\Phi(A)$.  The minimum distance expression for $\Phi(A)$ in Theorem~\ref{main.thm} can be seen as a refinement of the so-called {\em vertex-facet distance} defined by Beck and Shtern~\cite{BeckS15}.  \blue{More precisely if we specialize Beck and Shtern's construction of vertex-facet distance~\cite[Lemma 3.1]{BeckS15} to our context and assume a suitable normalization for the facet defining hyperplanes for $\conv(A)$,  then  it follows that the vertex-facet distance of the polytope $\conv(A)$ is
\[
\dmin_{F\in\facets(\conv(A))}
 \dist(\affine(F),\vertices(\conv(A\setminus F)))
\]
provided that $A = \vertices(\conv(A))$.

As the statement of Theorem~\ref{main.thm} shows, the quantity $\Phi(A)$ has a similar geometric expression.  As a consequence of Theorem~\ref{main.thm}, it follows that when $A = \vertices(\conv(A))$ the facial distance $\Phi(A)$ is at least as large as the vertex-facet distance of $\conv(A)$.
}

\begin{theorem}\label{main.thm}
Assume \blue{$A= \matr{a_1 & \cdots & a_n }\in \R^{m\times n}$ and at least two columns of $A$} are different.  Then
\[
\Phi(A) = \blue{\dmin_{F\in\faces(\conv(A)) \atop \emptyset \subsetneq F \subsetneq\conv(A)}}
 \dist(F,\conv(A\setminus F)).
\]
Furthermore, if $F\in\faces(\conv(A))$ minimizes the right hand side, then there exist $x,z\in \Delta_{n-1}$ such that $Az\in F, Ax \in \conv(A\setminus F)$ and 
\[
\Phi(A) = \Phi(A,x,z) = \dmax_{s\in S(x), \, a\in A} \ip{p}{s-a} = \|A(x-z)\|
\]
for $p = \frac{A(x-z)}{\|A(x-z)\|}.$ 
\end{theorem}
\begin{proof}
We first show that 
\begin{equation}\label{to.prove.thm1}
\Phi(A) \ge
\blue{\dmin_{F\in\faces(\conv(A)) \atop \emptyset \subsetneq F \subsetneq\conv(A)}}
\dist(F,\conv(A\setminus F)).
\end{equation}
To that end, assume $x,z\in \Delta_{n-1}$ are such that $A(x-z) \ne 0$ and let $d:=\frac{A(x-z)}{\|A(x-z)\|}$.  Let $p\in\R^m$ be a vector attaining the minimum in \eqref{eq.def.phi}.  Consider the face $F$ of 
$\conv(A)$ defined as
\[
F = \Argmin_{v\in \conv(A)} \ip{p}{v}.
\]
\blue{Observe that $\emptyset \ne F\ne \conv(A)$ because $\conv(A)$ is a nonempty compact set and $Ax\not \in F$.}  From Proposition~\ref{prop.restricted.width} it follows that $\Phi(A,x,z) = \|u-v\|$ for some $v\in F$ and $u\in\conv(A\setminus F).$  Therefore, 
$$\dist(F,\conv(A\setminus F)) \le \| u -  v\|  = \Phi(A,x,z).$$ 
Since this holds for any $x,z\in \Delta_{n-1}$ such that $A(x-z) \ne 0$, \blue{inequality \eqref{to.prove.thm1} follows.}

Next we show the reverse inequality 
\begin{equation}\label{to.prove.thm1.other}
\Phi(A) \le \blue{
\dmin_{F\in\faces(\conv(A)) \atop \emptyset \subsetneq F \subsetneq\conv(A)}} \dist(F,\conv(A\setminus F)).
\end{equation}
  To that end, assume $F$ minimizes \blue{the right-hand-side in \eqref{to.prove.thm1.other}}.  Let  $u \in \conv(A\setminus F)$ and $v\in F$ be such that 
\begin{equation}\label{eq.distance}
\dist(F,\conv(A\setminus F)) = \|u - v\|.
\end{equation}
\blue{
The optimality conditions for $u$ in \eqref{eq.distance} 
imply that $\ip{v-u}{s-u} \le 0$ for all $s\in \conv(A\setminus F)$. Likewise the optimality conditions for $v$ in \eqref{eq.distance} imply that $\ip{u-v}{t-v} \le 0$ for all $t\in F$.
Let $G \in \faces(\conv(A\setminus F))$ be the face defined by 
\[
G:=\{s\in \conv(A\setminus F): \ip{u-v}{s-u} = 0\}.
\]
By taking a face of $F$ if necessary, we can assume that 
$$\ip{u-v}{t-v} = 0 \text{ for all } t\in F.$$  
Therefore,}
\begin{equation}\label{eq.orthogonal}
\ip{u-v}{s-u} = \ip{u-v}{t-v} = 0\; \text{ for all } \; t\in F, s\in G.
\end{equation}
Next we claim that 
\begin{equation}\label{eq.acute}
\ip{u-v}{a-v} \ge 0 \; \text{ for all }\;a\in A.
\end{equation}  
We prove this claim by contradiction.  Assume $a\in A$ is such that $\ip{u-v}{a-v} < 0$. Then $\ip{u-v}{a-u} = \ip{u-v}{a-v} - \|u-v\|^2 < 0$ and from \eqref{eq.orthogonal} we get $a\not \in F$.  Hence for $\lambda > 0$ sufficiently small the point $u+\lambda(a-u) \in \conv(A\setminus F)$ satisfies
\[
\|u+\lambda(a-u) - v\|^2 = \|u-v\|^2 + 2\lambda\ip{u-v}{a-u} + \lambda^2\|v \|^2
< \|u-v\|^2, 
\]
which contradicts~\eqref{eq.distance}.  Thus \eqref{eq.acute} is proven.

Let $x,z \in \Delta_{n-1}$ be such that $u=Ax, v= Az$ and $S(x) = A\cap G$.  The latter is possible since $u\in G$.   We finish by observing that for $p=d=\frac{A(x-z)}{\|A(x-z)\|} = \frac{u-v}{\|u-v\|}$
\begin{align*}
\Phi(A,x,z) &\le \dmax_{s\in S(x), \, a\in A} \ip{p}{s-a} \\
&\le \dmax_{s\in G,a\in A} \ip{p}{s-a} \\&= \ip{p}{u-v} \\&= \|u-v\| \\
& = \dist(F,\conv(A\setminus F)).
\end{align*}
\blue{The first step follows from the construction of $\Phi(A,x,z)$.  The second step follows because $S(x) \subseteq G$.}
The third and \blue{fourth} steps follow from \eqref{eq.orthogonal} and \eqref{eq.acute} and the choice of $p$.  The \blue{fifth} step follows from~\eqref{eq.distance}.  
\blue{Finally, since  $\Phi(A) \le \Phi(A,x,z) \le \dist(F,\conv(A\setminus F))$ inequality \eqref{to.prove.thm1.other} follows. }
\qed
\end{proof}

\bigskip

From Theorem~\ref{main.thm} we can readily compute the values of $\Phi(A)$ in the special cases detailed in the examples below. We use the following notation.   Let $e \in \R^m$ denote the vector with all components equal to one, and for $i=1,\dots,m$ let $e_i\in \R^m$ denote the vector with $i$-th component equal to one and all others equal to zero.

\begin{example} Suppose $A = \{0,1\}^m \subseteq \R^m.$
By Theorem~\ref{main.thm}, induction, and symmetry it follows that
\[
\Phi(A) = \dist(0,\conv(A\setminus\{0\})) = \dist(0,\conv\{e_1,\dots,e_m\}) = \frac{\|e\|}{m}=\frac{1}{\sqrt{m}}.
\]

\end{example}
\begin{example}
Let $A = \{e_1,\dots,e_m\}\subseteq \R^m$.  By Theorem~\ref{main.thm} and the facial structure of $\conv(A)$ it follows that 
\begin{align*}
\Phi(A) &= \dmin_{\emptyset \subsetneq S \subsetneq A} 
\dist(\conv(S),\conv (A\setminus S)) \\
& = \min_{\emptyset \subsetneq S \subsetneq A} \left\| \frac{\sum_{s\in S}s}{\vert S\vert} - \frac{\sum_{a\in A\setminus S}a}{\vert A\setminus S\vert} \right \| \\
&= \min_{r\in\{1,\dots,m-1\}}\sqrt{\frac{m}{r(m-r)}} \\
&= \left\{\begin{array}{ll}
 \frac{2}{\sqrt{m}} & \text{ if } m \text{ is even } \\
 \frac{2}{\sqrt{m-\frac{1}{m}}}   & \text{ if } m \text{ is odd. } 
 \end{array} \right.
\end{align*}
\end{example}

We note that the values for $\Phi(A)$ in the above examples match exactly the values of the {\em pyramidal width} defined by Lacoste-Julien and Jaggi~\cite{LacoJ15}.
Theorem~\ref{main.thm.2} below shows that indeed the pyramidal width is identical to the condition measure $\Phi(A)$.  As a byproduct of this identity,  the original definition of pyramidal width given in~\cite{LacoJ15} can be simplified.

Lacoste-Julien and Jaggi define the {\em pyramidal directional width} of a finite set $A\subseteq \R^m$ with respect to a direction $r\in\R^m \setminus\{0\}$ and a base point $u \in \conv(A)$ as
\[
\pyrdir(A,r,u) := \dmin_{S\in S_u} \dmax_{a\in A, s\in S} \ip{\frac{r}{\|r\|}}{a-s}
\]
where $S_u = \{S\subseteq A: u \in \conv(S)\}.$
Lacoste-Julien and Jaggi also define the {\em pyramidal  width} of a set $A$ as
\[
\pyrwidth(A):= 
\dmin_{F\in \faces(\conv(A))\atop u \in F, r \in \cone(F-u)\setminus\{0\}}
 \pyrdir(F\cap A, r,u).
\]
Observe that $r \in \cone(F-u)\setminus\{0\}$ if and only if $r$ is a positive multiple of some $v-u$ where $u,v\in F$ and $u-v \ne 0$.  Therefore
\begin{equation}\label{eq.pyrwidth.def}
\pyrwidth(A) = \dmin_{ F\in\faces(\conv(A))\atop u,v\in F, u\ne v} \pyrdir(F\cap A, v-u,u).
\end{equation}
\blue{
We note that in the original definition of the pyramidal directional width in \cite{LacoJ15} Lacoste-Julien and Jaggi use the following more restricted set $\tilde S_u$ instead of $S_u$   $$\tilde S_u = \{S\subseteq A: u \text{ is a proper convex combination of elements in } \; S\}.$$
The larger set $S_u$ that we use above yields an equivalent definition of pyramidal directional width because for all $u\in\conv(A)$ and $r\in\R^m \setminus\{0\}$
\[
\dmin_{S\in S_u} \dmax_{a\in A, s\in S} \ip{\frac{r}{\|r\|}}{a-s} = \dmin_{S\in \tilde S_u} \dmax_{a\in A, s\in S} \ip{\frac{r}{\|r\|}}{a-s}.
\]
}

\begin{theorem}\label{main.thm.2}
Assume \blue{$A= \matr{a_1 & \cdots & a_n }\in \R^{m\times n}$ and at least two columns of $A$} are different.  Then
\begin{equation}\label{eq.pyrwidth}
 \Phi(A) =  \min_{u,v\in\conv(A), u \ne v} \pyrdir(A, v-u,u) = \pyrwidth(A).
\end{equation}
\end{theorem}
\begin{proof}
Assume $F\in\faces(A)$ minimizes $\dist(F,\conv(A\setminus F))$
and $x,z\in \Delta_{n-1}$ are \blue{chosen as in the second statement of Theorem~\ref{main.thm} so that for $p = \frac{A(x-z)}{\|A(x-z)\|}$ we have
\[
\Phi(A) = \Phi(A,x,z) = \dmax_{s\in S(x), \, a\in A} \ip{p}{s-a}. 
\]
}
Therefore for $u=Ax$ and  $v= Az$ we have $S(x) \in S_u$ and
\begin{align*}
\Phi(A) &= \dmax_{a\in A, s\in S(x)} \ip{\frac{v-u}{\|v-u\|}}{a-s} \\ &\ge 
\dmin_{S\in S_u} \dmax_{a\in A, s\in S} \ip{\frac{v-u}{\|v-u\|}}{a-s}\\
&= \pyrdir(A,v-u,u).
\end{align*}
Hence we conclude that $ \Phi(A) \ge \dmin_{u,v\in\conv(A), u \ne v} \pyrdir(A, v-u,u).$  

On the other hand, for $u,v\in\conv(A), u \ne v$ let $S \in S_u$ be such that
\[
\pyrdir(A,v-u,u) :=  \dmax_{a\in A, s\in S} \ip{\frac{v-u}{\|v-u\|}}{a-s}.
\]
\blue{Since $S\in S_u$, we have $u \in \conv(S)$ and thus there exists $x\in \Delta_{n-1}$ such that $S(x) \subseteq S$ and $Ax = u$.  Let $z\in \Delta_{n-1}$ be such that $Az = v$.  }
Taking  $p=d = \frac{A(x-z)}{\|A(x-z)\|} = -\frac{v-u}{\|v-u\|}$ it follows that 
\begin{align*}
\Phi(A) \le \Phi(A,x,z) 
&\le \dmax_{s\in S(x), \, a\in A} \ip{p}{s-a} \\
&=\dmax_{a\in A, s\in S(x)} \ip{\frac{v-u}{\|v-u\|}}{a-s} \\
&\le\dmax_{a\in A, s\in S} \ip{\frac{v-u}{\|v-u\|}}{a-s} \\
&= \pyrdir(A,v-u,u).
\end{align*}
Consequently $ \Phi(A) \le \dmin_{u,v\in\conv(A), u \ne v} \pyrdir(A, v-u,u)$ as well.    Therefore the identity between the first two terms in \eqref{eq.pyrwidth} follows.  
To finish, observe that by~\eqref{eq.pyrwidth.def}
\begin{align*}
\pyrwidth(A) &= \dmin_{F\in \faces(\conv(A))} \min_{u,v\in\conv(F), u \ne v} \pyrdir(F\cap A, v-u,u) \\ &= \dmin_{F\in \faces(\conv(A)) \atop \exists u,v\in F, \; u\ne v } \Phi(F\cap A) \\ &= \Phi(A).
\end{align*}
The second step follows by applying the identity between the first two terms in \eqref{eq.pyrwidth} to each $F\cap A$.  The third step follows because  \blue{ on the one hand $\conv(A) \in \faces(\conv(A))$ implies 
$$
\dmin_{F\in \faces(\conv(A)) \atop \exists u,v\in F, \; u\ne v } \Phi(F\cap A) \le \Phi(A);
 $$
 and on the other hand Theorem~\ref{main.thm} implies
 $$
\dmin_{F\in \faces(\conv(A)) \atop \exists u,v\in F, \; u\ne v } \Phi(F\cap A) \ge \Phi(A).
 $$
\qed}
\end{proof}

\bigskip

As we will see in Section~\ref{sec.FrankWolfe} below, the following {\em localized} variant of $\Phi(A)$, which could be quite a bit larger than $\Phi(A)$ plays a role in the linear  convergence rate of the Frank-Wolfe algorithm.   Assume $A= \matr{a_1&\cdots&a_n}\subseteq \R^{m\times n}$ and at least two columns of $A$ are different. For $z \in \Delta_{n-1}$ let
\[
\Phi(A,z) := \min_{x\in \Delta_{n-1}: A(x-z) \ne 0} \Phi(A,x,z),
\]
and for \blue{nonempty} $Z \subseteq \Delta_{n-1}$ let
\[
\Phi(A,Z) := \inf_{z\in Z} \Phi(A,z).
\]
Theorem~\ref{prop.lower.bound} gives a localized version of Theorem~\ref{main.thm} for $\Phi(A,Z)$.  Its proof relies on the following localized version of Proposition~\ref{prop.restricted.width}.  We omit the proof of Proposition~\ref{prop.modified.width} as it is a straightforward modification of the proof of Proposition~\ref{prop.restricted.width}.
\begin{proposition}\label{prop.modified.width} Assume 
\blue{$A =\matr{a_1&\cdots&a_n} \in \R^{m\times n}$ and} $x,z\in \Delta_{n-1}$ are such that $A(x-z) \ne 0$ and  let $d:=\frac{A(x-z)}{\|A(x-z)\|}$. If $F\in \faces(\conv(A))$ contains $Az$ then \begin{multline*}
\min_{p:\ip{p}{d} = 1} \max_{s\in S(x), a \in F\cap A} \ip{p}{s - a} =
\\
\max\left\{\lambda > 0: \exists w,y\in \Delta_{n-1}, \; I(w) \subseteq I(x), Ay\in F,\; A(w - y) = \lambda d\right\}.
\end{multline*}
Furthermore, if $p$ minimizes the left hand side  and $u=Aw, v = Ay$ maximize the right hand side  then $v \in \conv(B)$ and $u \in \conv(A\setminus B)$ where $B = \displaystyle\Argmin_{a\in F\cap A} \ip{p}{a}$\blue{, and \[\min_{p:\ip{p}{d} = 1} \max_{s\in S(x), a \in F\cap A} \ip{p}{s - a} = \|u-v\|.\]}

\end{proposition}

\begin{theorem} \label{prop.lower.bound} Assume \blue{$A= \matr{a_1 & \cdots & a_n }\in \R^{m\times n}$ and at least two columns of $A$} are different.  Then for \blue{nonempty} $Z\subseteq \Delta_{n-1}$ 
\[
\Phi(A,Z) \ge \dmin_{G\in\faces(F)\atop \blue{\emptyset\ne G\ne \conv(A)}} \dist(G,\conv(A\setminus G))
\]
where $F$ is the smallest face of $\conv(A)$ containing $AZ = \{Az: z\in Z\}$.  
\end{theorem}
\begin{proof}
This is a straightforward modification of the proof of Theorem~\ref{main.thm}.   Observe that for $x\in \Delta_{n-1}, z\in Z$ \blue{with $A(x-z) \ne 0$} and $d = \frac{A(x-z)}{\|A(x-z)\|}$
\[
\Phi(A,x,z) =  \blue{\min_{p:\ip{p}{d} = 1} \max_{s\in S(x), a \in A} \ip{p}{s - a}}
\ge  \min_{p:\ip{p}{d} = 1} \max_{s\in S(x), a \in F\cap A} \ip{p}{s - a}.
\]
Let $p\in\R^m$ be a vector attaining the minimum in the above right hand side and consider the face $G$ of 
$F$ defined as
\[
G = \Argmin_{v\in F} \ip{p}{v}.
\]
\blue{Observe that $\emptyset \ne G \ne \conv(A)$ because $F$ is a nonempty compact set and $Ax \not \in G.$}
From Proposition~\ref{prop.modified.width} it follows that for some $u \in \conv(A\setminus G)$ and $v \in G$
\begin{align*}
\Phi(A,x,z) &=\blue{\min_{p:\ip{p}{d} = 1} \max_{s\in S(x), a \in A} \ip{p}{s - a}} \\
&\ge  \min_{p:\ip{p}{d} = 1} \max_{s\in S(x), a \in F\cap A} \ip{p}{s - a} \\
&= \|u-v\| \\
&\ge \dist(G,\conv(A\setminus G)).
\end{align*}
Since this holds for all $x\in \Delta_{n-1}, z\in Z$, we conclude that $$\Phi(A,Z) \ge \dmin_{G\in\faces(F)} \dist(G,\conv(A\setminus G)).$$
\qed
\end{proof}

It is evident from Theorem~\ref{main.thm} and Theorem~\ref{prop.lower.bound} that $\Phi(A,Z)$ can be arbitrarily larger than $\Phi(A)$. \blue{In particular, consider 
$A = \matr{M & 0 & 0 & \cdots &0 \\ \frac{1}{2} & \frac{1}{2} & \frac{1}{3} &\cdots & \frac{1}{n}} \in \R^{2 \times n}$ with $M \gg 1$, and $Z=\{e_1\}\subseteq \Delta_{n-1}$.  For this  $A$ and $Z$ we have 
$$
\Phi(A,Z) \ge M \gg \frac{1}{n}-\frac{1}{n-1} =\dmin_{F\in\faces(\conv(A))\atop \emptyset \subsetneq F \subsetneq\conv(A)}\dist(F,\conv(A\setminus F))
$$
because in this case $\{a_1\} = \{Ae_1\} =AZ$ is the smallest face of $\conv(A)$ containing $AZ$ and the minimum in the right hand side is attained at the face $F=\{a_n\}$ of $\conv(A)$.}

\section{Frank Wolfe algorithm with away steps}
\label{sec.FrankWolfe}
We next present a fairly generic linear convergence result for a version of the Frank Wolfe algorithm with away steps.  We emphasize that the  statement in Theorem~\ref{thm.fw.linear} can be found in or inferred from results already shown in \cite{BeckS15,LacoJ15}. 
The goal of this section is to highlight three central premises that yield the proof of linear convergence, namely a first premise in the form of a {\em slope bound}, a second premise in the form of a {\em decrease condition}, and a third premise on the {\em search direction} selected by the algorithm.  As we detail below, the third premise is naturally tied to the condition measures $\Phi(A)$ and $\Phi(A,Z)$ discussed in Section~\ref{sec.main}.

Assume $A =\matr{a_1 & \cdots & a_n}\in \R^{m\times n}$ and consider the problem
\begin{equation}\label{min.problem}
\min_{\blue{u}\in \text{conv}(A)} f(\blue{u}) 
\end{equation}
where $f:\conv(A) \rightarrow \R$ is a differentiable convex function.  \blue{Throughout this section we assume that at least two columns of $A$ are different as otherwise problem~\eqref{min.problem} is trivial.}

\medskip

We will rely on the following notation related to problem~\eqref{min.problem} throughout the rest of the paper.  Let $f^*, \blue{U^*}, Z^*$ be defined as
\[
f^* := \dmin_{\blue{u}\in \conv(A)} f(\blue{u}), \; \blue{U^*}:=\Argmin_{\blue{u}\in \conv(A)} f(\blue{u}), \; Z^*:=\{z\in \Delta_{n-1}: Az\in \blue{U^*}\}.
\]

\begin{algorithm}
  \caption{Frank-Wolfe Algorithm with Away Steps
    \label{alg:frankwolfe}}
  \begin{algorithmic}[1]
\State Pick $x_0 \in \Delta_{n-1}$; put $\blue{u}_0:= Ax_0;\; k:=0$
 \For{$k=0,1,2,\dots$}
\State $j := \displaystyle\argmin_{i=1,\dots,n} \ip{\nabla f(\blue{u}_k)}{a_i};\; \ell := \displaystyle\argmax_{i\in I(x_k)}  \ip{\nabla f(\blue{u}_k)}{a_i}$
 \If{$\ip{\nabla f(\blue{u}_k)}{a_j-\blue{u}_k} < \ip{\nabla f(\blue{u}_k)}{\blue{u}_k-a_{\ell}}$ \blue{or $\vert I(x_k)\vert =1$}} 
\State $v:=a_j - \blue{u}_k;\; \blue{w}:=e_j - x_k; \; \blue{\gamma}_{\max} := 1\;\;\;\;\,$  \quad (regular step)
\quad \Else 
$\;\;$
\State  $v:= \blue{u}_k - a_\ell; \; \blue{w}:= x_k -e_{\ell}; \; \blue{\gamma}_{\max} := \frac{\ip{e_\ell}{x_k}}{1-\ip{e_\ell}{x_k}}\;\;$  (away step) 
\EndIf 
\State  choose $\blue{\gamma}_k \in [0,\blue{\gamma}_{\max}]$ 
\State $x_{k+1} := x_k + \blue{\gamma}_k \blue{w}; \; \; \blue{u}_{k+1} := \blue{u}_k + \blue{\gamma}_k v = Ax_{k+1}$ 
\EndFor
  \end{algorithmic}
\end{algorithm}

Theorem~\ref{thm.fw.linear} gives a fairly generic linear convergence result for Algorithm~\ref{alg:frankwolfe}.  This result hinges on the following three premises.

\medskip

\noindent{\bf Premise 1:}
{\em  There exists $\mu > 0$ such that the objective function $f$ satisfies the following  bound: For all  $\blue{u}\in \conv(A)$ and $\blue{u^*} \in \blue{U^*}$ \[
\blue{\ip{\nabla f(\blue{u})}{\blue{u}-\blue{u^*}} \ge \|\blue{u}-\blue{u^*}\|\sqrt{2\mu(f(\blue{u}) - f^*)}.}
\]
}

Premise 1 readily holds if $f$ is strongly convex with parameter $\mu$.  In this case  the optimality of $u^*$, the strong convexity of $f$, and the arithmetic-geometric inequality imply
\begin{align*}
\ip{\nabla f(\blue{u})}{\blue{u}-\blue{u}^*} &\ge f(y) - f^* +\frac{\mu}{2}\|\blue{u}-\blue{u}^*\|^2\ge \|\blue{u}-\blue{u}^*\| \sqrt{2\mu(f(\blue{u}) - f^*)}.
\end{align*}
From the error bound of Beck and Shtern~\cite[Lemma 2.5]{BeckS15}, it follows  that Premise 1 also holds in the more general case when $f(u) = \blue{h}(Eu) + \ip{b}{u}$ for some strongly convex function $\blue{h}$.

\medskip

\noindent{\bf Premise 2:} {\em There exists $L>0$ such that \blue{for all $\gamma\in [0,\gamma_{\max}]$
\[
f(u_k + \gamma v) \le f(u_k) + \ip{\nabla f(\blue{u}_k)}{\gamma v} + \frac{L\gamma^2\|v\|^2}{2},
\]
and the steplength $\blue{\gamma}_k$ at each iteration of Algorithm~\ref{alg:frankwolfe} satisfies 
\[
\blue{\gamma}_k = \argmin_{\blue{\gamma} \in [0,\blue{\gamma}_{\max}]} \left\{\ip{\nabla f(\blue{u}_k)}{\gamma v} + \frac{L\gamma^2\|v\|^2}{2}
\right\} = \min\left\{-\frac{\ip{\nabla f(\blue{u}_k)}{v}}{L\|v\|^2},\blue{\gamma}_{\max}\right\}.
\]}}
Premise  2  holds if $\nabla f$ is Lipschitz with constant $L$.

\medskip

\noindent{\bf Premise 3:} {\em There exists $c>0$ such that the search direction $v$ 
selected in Step 5 or Step 7 of Algorithm~\ref{alg:frankwolfe} satisfies 
\[
-\frac{\ip{\nabla f(\blue{u}_k)}{v}}{\ip{\nabla f(\blue{u}_k)}{d}} \ge c
\]
for all $d = \blue{\frac{u_k-u^*}{\|u_k-u^*\|}}$ such that $\blue{u}^*\in U^*$ and $\ip{\nabla f(\blue{u}_k)}{d} > 0.$
}

\medskip

Premise 3 holds at each iteration of the Frank-Wolfe Algorithm with Away Steps for $c = \frac{\Phi(A)}{2}$ as well as for the sharper bound $c = \frac{\Phi(A,Z^*)}{2}.$ To see this, observe that \blue{as noted by \cite{LacoJ15}}, the choice between regular or away steps ensures 
\[
-\ip{\nabla f(\blue{u}_k)}{v} \ge \frac{1}{2} \ip{\nabla f(\blue{u}_k)}{a_\ell-a_j} > 0.
\]
Let $z\in Z^*$ be such that $Az = \blue{u}^* \in \blue{U^*}.$  \blue{Take $d = \frac{A(x_k-z)}{\|A(x_k-z)\|}=\frac{\blue{u}_k-u^*}{\|\blue{u}_k-u^*\|}$ and  $p = \frac{\nabla f(\blue{u}_k)}{\ip{\nabla f(\blue{u}_k)}{d}}.$  Observe that 
$\ip{p}{d} =1$ and thus from the construction of $\Phi(A,x,z)$ we get}
\[
\frac{\ip{\nabla f(\blue{u}_k)}{ a_\ell - a_j}}{\ip{\nabla f(\blue{u}_k)}{d}} = \dmax_{\ell\in I(x_k),a\in A } \ip{p}{a_\ell - a} \ge \Phi(A,x_k,z) \ge \Phi(A,Z^*).
\]
Hence \[
-\frac{\ip{\nabla f(\blue{u}_k)}{v}}{\ip{\nabla f(\blue{u}_k)}{d}} \ge \frac{\ip{\nabla f(\blue{u}_k)}{a_\ell-a_j}}{2\ip{\nabla f(\blue{u}_k)}{d}} \ge \frac{\Phi(A,Z^*)}{2} \ge \frac{\Phi(A)}{2}.
\]
\blue{We note that a similar, though a bit weaker, bound for $c$ is given in \cite{BeckS15} in terms of the vertex-facet distance of $\conv(A)$. (See \cite[Corollary 3.1]{BeckS15}.)}

\begin{theorem}\label{thm.fw.linear} Assume $x_0\in\Delta_{n-1}$ in Step 1 of Algorithm~\ref{alg:frankwolfe} is a vertex of $\Delta_{n-1}$.  If Premise 1, Premise 2, and Premise 3 hold then the sequence of points $\{\blue{u}_k: k =0,1,\dots\}$ generated by Algorithm~\ref{alg:frankwolfe} satisfies
\begin{equation}\label{lin.conv}
f(\blue{u}_k) - f^* \le \left(1-r\right)^{k/2} (f(\blue{u}_0) - f^*)
\end{equation}
for $r = \blue{\min\left\{\frac{\mu c^2}{L\cdot\diam(A)^2},\frac{1}{2}\right\}}$.  In particular, \blue{if $\mu \le L$ then }\eqref{lin.conv} holds for the solution-independent rate
$$r =  \frac{\mu}{L} \cdot \frac{\Phi(A)^2}{4\diam(A)^2}$$
as well as for the sharper, though solution-dependent rate
$$r =  \frac{\mu}{L} \cdot \frac{\Phi(A,Z^*)^2}{4\diam(A)^2}.$$
\end{theorem}
\begin{proof}  This proof is an adaptation of the proofs in \cite{BeckS15,LacoJ15,PenaRS15}. \blue{We consider the following three possible cases separately: $\gamma_k < \gamma_{\max}, \gamma_k = \gamma_{\max}\ge 1$, and $\gamma_k = \gamma_{\max} < 1$.  }

\noindent
\blue{{\bf Case 1:} $\blue{\gamma}_k < \blue{\gamma}_{\max}$.} Premise 2 and the choice of $v$ imply that
\begin{equation}\label{eq.lin.conv.1}
f(\blue{u}_{k+1}) \le f(\blue{u}_k) - \frac{\ip{\nabla f(\blue{u}_k)}{v}^2}{2L\|v\|^2}.
\end{equation}
Premise 3 and Premise 1 in turn yield
\[
\ip{\nabla f(\blue{u}_k)}{v}^2 \ge \frac{c^2\ip{\nabla f(\blue{u}_k)}{\blue{u}_k-\blue{u}^*}^2}{\|\blue{u}_k-\blue{u}^*\|^2} \\
\ge 2\mu c^2(f(\blue{u}_k)-f^*).
\]
Plugging the last inequality into \eqref{eq.lin.conv.1} we get
\begin{align*}
f(\blue{u}_{k+1}) - f^* &\le \left(1 - \frac{\mu c^2 }{L\|v\|^2}\right) (f(\blue{u}_k) - f^*) \\
&\le\left(1 - \frac{\mu c^2 }{L\cdot\diam(A)^2}\right) (f(\blue{u}_k) - f^*).
\end{align*}

\noindent
\blue{{\bf Case 2:} $\gamma_k = \gamma_{\max} \ge 1$.   Premise 2 and the choice of $v$ imply that
\[
f(u_{k+1}) = f(u_k+\gamma_{\max}v) \le f(u_k) + \frac{\gamma_{\max}}{2}\ip{\nabla f(u_k)}{v} \le f(u_k) + \frac{1}{2}\ip{\nabla f(u_k)}{v}.
\]
On the other hand, the choice of $v$ and the convexity of $f$ yield 
$$\ip{\nabla f(u_k)}{v} \le \min_{u\in\conv(A)} \ip{\nabla f(u_k)}{u-u_k} \le \ip{\nabla f(u_k)}{u^*-u_k} \le f^* - f(u_k).$$
Consequently, when $\gamma_k = \gamma_{\max} \ge 1$ we have
\[
f(u_{k+1}) - f^* \le f(u_k)-f^* + \frac{1}{2}(f^*-f(u_k))= \frac{1}{2} (f(u_k)-f^*).
\]

\noindent
\blue{{\bf Case 3:} $\blue{\gamma}_k = \blue{\gamma}_{\max}< 1.$}  Premise 2 and the choice of $v$ imply that
\[
f(u_{k+1}) = f(u_k+\gamma_{\max}v) \le f(u_k) + \frac{\gamma_{\max}}{2}\ip{\nabla f(u_k)}{v} \le f(u_k).
\]
Hence  $f(\blue{u}_{k+1}) - f^* \le  f(\blue{u}_{k}) - f^*.$
}

\medskip

\blue{To finish, it suffices to show that after $N$ iterations, Case 3 occurs at most $N/2$ times. }
To that end, we apply the following clever argument from \cite{LacoJ15}.
 Each time $\blue{\gamma}_k = \blue{\gamma}_{\max}< 1$ we have $\vert I(x_{k+1})\vert  \le \vert I(x_k)\vert -1$.  On the other hand, each time $\blue{\gamma}_k < \blue{\gamma}_{\max}$ we have $\vert I(x_{k+1})\vert \le \vert I(x_k) \vert + 1$ and each time $\blue{\gamma}_k = \blue{\gamma}_{\max}\ge 1$ we have $\vert I(x_{k+1})\vert \le \vert I(x_{k})\vert.$ Since $\vert I(x_0)\vert  = 1$ and $\vert I(x_k)\vert  \ge 1$ for all $x_k\in\Delta_{n-1}$, it follows that after $N$ iterations there must have been at least as many iterations with $\blue{\gamma}_k < \blue{\gamma}_{\max}$
as there were  with $\blue{\gamma}_k = \blue{\gamma}_{\max}<1$.  In particular, the number of iterations with $\blue{\gamma}_k = \blue{\gamma}_{\max}<1$ is at most $N/2$.  
\qed
\end{proof}

\bigskip

As we noted at the end of Section~\ref{sec.main}, the quantity $\Phi(A,Z^*)$  in Theorem~\ref{thm.fw.linear} can be arbitrarily better (larger) than $\Phi(A).$   Observe that the solution-independent bound is simply the most conservative solution-dependent one for all possible solution sets $Z^*\subseteq \Delta_{n-1}$.  

\blue{The linear convergence rate \eqref{lin.conv} with $r = \frac{\mu}{L}\cdot\frac{\Phi(A)^2}{4\diam(A)^2}$  matches the one given in~\cite[Theorem 1]{LacoJ15} when $f$ is strongly convex with Lipschitz gradient.  Furthermore, the rate \eqref{lin.conv} in terms of Premises 1 through 3 is essentially equivalent to the one derived~\cite[Theorem 3.1]{BeckS15}.  A subtle difference is that the latter uses a weaker bound for Premise 3 in terms of the vertex-facet distance~\cite[Corollary 3.1]{BeckS15}.}


\bigskip

\section{Convex quadratic objective}
\label{sec.quadratic}
We next  \blue{establish a} linear convergence result \blue{similar to} Theorem~\ref{thm.fw.linear} for the case when the objective function is of the form \[
f(\blue{u}) = \frac{1}{2} \ip{\blue{u}}{Q \blue{u}} + \ip{b}{\blue{u}},\]
for an $m\times m$ symmetric positive semidefinite matrix $Q$ and $b\in \R^m$. Consider  problem~\eqref{min.problem} and Algorithm~\ref{alg:frankwolfe} for this objective function.  In this case the steplength $\blue{\gamma}_k$ in Step 9 of Algorithm~\ref{alg:frankwolfe} can be easily computed via the following exact line-search:
\begin{align}\label{eq.step.length}
\blue{\gamma}_k &:= \argmin_{\blue{\gamma} \in [0,\blue{\gamma}_{\max}]} f(\blue{u}_k + \blue{\gamma} v) \notag\\
&= \left\{ 
\begin{array}{ll} \min\left\{\blue{\gamma}_{\max}, - \frac{\ip{Q\blue{u}_k + b}{v}}{\ip{v}{Q v}} \right\} & \text{ if } \; \ip{v}{Q v}>0\\ \blue{\gamma}_{\max} & \text{ if } \; \ip{v}{Q v}=0. \end{array} \right.
\end{align}
In this case linear convergence can be obtained from Theorem~\ref{thm.fw.linear} and the error bound~\cite[Lemma 2.5]{BeckS15}.  However, the  more explicit rate of convergence \blue{stated in Theorem~\ref{thm.quad.conv} below} can be obtained via a refinement of the proof of Theorem~\ref{thm.fw.linear}.  

\blue{
Theorem~\ref{thm.quad.conv} is similar in spirit to the main result of Beck and Teboulle~\cite{BeckT04}. However, the treatment in~\cite{BeckT04} is concerned with the regular Frank-Wolfe algorithm for minimizing a quadratic objective function of the form $\|Mu - b\|^2$ over a compact convex domain. Their main linear convergence result in \cite{BeckT04}, namely~\cite[Proposition 3.2]{BeckT04}, requires that the minimizer is in the relative interior of the domain and that its objective value is zero.  By contrast, Theorem~\ref{thm.quad.conv} below does not require any assumption about the location of the minimizer or its objective value and applies to  Algorithm~\ref{alg:frankwolfe}, that is, the variant of Frank-Wolfe algorithm with away steps.
}


\medskip

\blue{
The statement and proof of Theorem~\ref{thm.quad.conv} are modifications of those of Theorem~\ref{thm.fw.linear}.  The linear convergence rate in Theorem~\ref{thm.quad.conv} is stated in terms of a variant $\bar \Phi_g$ of $\Phi$.  Similar to the proof of Theorem~\ref{thm.fw.linear}, the proof Theorem~\ref{thm.quad.conv} follows by putting together variants of the three premises described in Section~\ref{sec.FrankWolfe}.

The construction of the variant $\bar \Phi_g$ of $\Phi$ relies on the objects $Z(g)$ and $\|\cdot\|_g$ defined next.  Assume $ \bar A \in \R^{(m+1)\times n}$.  For $g \in \R^m$ define $Z(g)\subseteq \Delta_{n-1}$ 
 as follows 
\[
Z(g) := \Argmin_{z\in \Delta_{n-1}} \ip{\matr{g\\1}}{\bar Az}.
\]
For $\bar v = \matr{v \\ v_{m+1}} \in \R^{m+1}$ let
\[
\|\bar v\|_g = \sqrt{\left\|\matr{I_m & 0}\bar v\right\|^2 + \left\vert\matr{g\transp & 1} \bar v \right\vert } = \sqrt{\|v\|^2 + \vert g\transp v + v_{m+1}\vert}.
\]
Observe that $ \|\bar v\|_g > 0$ if $\bar v \ne 0$.

For $x\in \Delta_{n-1}$ and $z \in Z(g)$ with $\bar A(x-z) \ne 0$ let $\bar d:=\frac{\bar A (x-z)}{\|\bar A(x-z)\|_g}$ and define \begin{equation}\label{eq.def.bar.phi}
\bar \Phi_g\left(\bar A,x,z\right) = 
\min_{p: \ip{p}{\bar d} = 1} \max\left\{\ip{p}{\bar a_\ell - \bar a_j}: \ell \in I(x),j\in\{1,\dots,n\}\right\}
\end{equation}
and
\[
\bar \Phi_g\left(\bar A\right) := \min_{x\in \Delta_{n-1},z\in Z(g)\atop \bar A(x-z) \ne 0} \bar \Phi_g\left(\bar A,x,z\right).
\]
The construction of $\bar \Phi_g$ in \eqref{eq.def.bar.phi} resembles that of $\Phi$ in \eqref{eq.def.phi}.  The key difference is the type of normalization used in $\bar d$ that discriminates between the first $m$ and the last components of $\bar A(x-z)$ via $g$.  
As Proposition~\ref{prop.bar.phi} below shows, $\bar \Phi_g$ can be bounded below in terms of $\Phi$.  In particular, Proposition~\ref{prop.bar.phi} shows that $\bar \Phi_g(\bar A) > 0$ if at least two columns of $\bar A$ are different.

}

\begin{theorem}\label{thm.quad.conv} 
Assume the objective function $f(\blue{u})$ in problem~\eqref{min.problem} is $f(\blue{u}) = \frac{1}{2}\ip{\blue{u}}{Q\blue{u}} + \ip{b}{\blue{u}}$ where $Q$ is an $m \times m$ symmetric positive semidefinite matrix and $b\in\R^m$.  
\blue{Assume the matrix $\bar A = \matr{Q^{1/2}A\\b\transp A}$  has at least two different columns as otherwise problem~\eqref{min.problem} is trivial.

If $x_0\in\Delta_{n-1}$ in Step 1 of Algorithm~\ref{alg:frankwolfe} is a vertex of $\Delta_{n-1}$ and the steplength $\blue{\gamma}_k$ in Step 9 of Algorithm~\ref{alg:frankwolfe} is computed as in~\eqref{eq.step.length} then the convergence rate \eqref{lin.conv} holds with
\[
r = \min\left\{\frac{\bar \Phi_{g}\left(\bar A\right)^2}{8\diam(Q^{1/2}A)^2},\frac{1}{2}\right\},
\]
where 
$g = Q^{1/2}Az$ for $z\in Z^*$, which does not depend on the specific choice of  $z\in Z^*$.
}
\end{theorem}
\blue{
\begin{proof}
This is a modification of the proof of Theorem~\ref{thm.fw.linear}.  Let $u^* = Az \in U^*$ for $z\in Z^*$ and let $\blue{u}_k = Ax_k$ denote the $k$-th iterate generated by  Algorithm~\ref{alg:frankwolfe}.  The proof is a consequence of $Z^*\subseteq Z(g)$ and the following three premises that hold provided $\blue{u}_k \ne u^*$.

\bigskip

\noindent
{\bf Premise 1':} 
\[
\ip{\nabla f(\blue{u}_k)}{\blue{u}_k-u^*} \ge \|\bar A(x_k-z)\|_g\sqrt{f(\blue{u}_k) - f^*}.
\]

\bigskip

\noindent
{\bf Premise 2':} If $\gamma_k =- \frac{\ip{\nabla f(u_k)}{v}}{\ip{v}{Q v}}=- \frac{\ip{Q\blue{u}_k + b}{v}}{\ip{v}{Q v}} < \blue{\gamma}_{\max}$ then
\[
f(\blue{u}_k)-f(\blue{u}_{k+1}) = \frac{\ip{\nabla f(\blue{u}_k)}{v}^2}{2\ip{v}{Q v}}.
\]

\bigskip

\noindent
{\bf Premise 3':} 
\[
-\ip{\nabla f (u_k)}{v} \ge
\frac{\dmax_{\ell\in I(x_k),j\in\{1,\dots,n\}} \ip{\nabla f (u_k)}{a_\ell - a_j}}{
2} \ge \frac{\bar\Phi_g(\bar A)}{2} \cdot\frac{\ip{\nabla f(u_k)}{u_k-u^*}}{\|\bar A(x_k-z)\|_g}.
\]

\bigskip

Indeed, by putting together Premise 2', Premise 3', and Premise 1' it follows that when $\blue{\gamma}_k =- \frac{\ip{v}{Qu_k + b}}
{\ip{v}{Q v}} < \blue{\gamma}_{\max}$
\begin{align*}
f(u_k)-f(u_{k+1})  & = \frac{\ip{\nabla f(u_k)}{v}^2}{2 \ip{v}{Qv}} \\
& \ge \frac{\bar \Phi_{g}(\bar A)^2}{8 \diam(Q^{1/2}A)^2} \cdot \frac{\ip{\nabla f(u_k)}{u_k-u^*}^2}{\|\bar A(x_k-z)\|_g^2} \\
& \ge \frac{\bar \Phi_{g}(\bar A)^2}{8\diam(Q^{1/2}A)^2}\cdot(f(u_k) - f^*),
\end{align*}
and consequently
\[
f(u_{k+1})-f^*  \le \left(1-\frac{\bar \Phi_{g}\left(\bar A\right)^2}{8\diam(Q^{1/2}A)^2}\right)\cdot(f(u_k) - f^*).
\]
The rest of the proof follows exactly as the last two paragraphs of the proof of Theorem~\ref{thm.fw.linear}.

We next show that $g=Q^{1/2}Az$ does not depend on the specific choice of $z \in Z^*$ and $Z^*\subseteq Z(g)$. 

Let $z,z'\in Z^*$.  From the optimality of $u^*=Az$ and $u' = Az'$ we get both
$
0 \le \ip{Qu^*+b}{u'-u^*}$ and $0 \le \ip{Qu'+b}{u^*-u'}.
$
Hence
\[
0\le -\ip{Q(u^*-u')}{u^*-u'}=-\|Q^{1/2}(u^*-u')\|^2 \Rightarrow Q^{1/2} u^* =  Q^{1/2} u'.
\]
That is, $Q^{1/2}Az=Q^{1/2}Az'$.  In other words, $g = Q^{1/2}Az$ does not depend on the specific choice of $z\in Z^*.$

Next we show $Z^* \subseteq Z(g)$.  To that end, let $z\in Z^*$ and $u^* = Az$.  The optimality of $u^*$ implies that for all $u = Ax\in \conv(A)$
\begin{align*}
0 &\le \ip{\nabla f(u^*)}{u-u^*} \\ &
= \ip{Qu^*+b}{u-u^*} \\ & =
\ip{Q^{1/2}Az}{Q^{1/2}(Ax-Az)} + \ip{b}{Ax-Az}
\\ & = \ip{g}{Q^{1/2}Ax} + \ip{b}{Ax} - \left(\ip{g}{Q^{1/2}Az} + \ip{b}{Az}\right).
\end{align*}
Hence $z\in Z(g)$.  Since this holds for all $z\in Z^*$ we have $Z^*\subseteq Z(g)$.

\medskip

We next show each of the above three premises.  Observe that 
\[
\ip{\nabla f(u_k)}{u_k - u^*} = f(u_k) - f^* + \frac{1}{2}\ip{u_k - u^*}{Q(u_k-u^*)} \ge f(u_k) - f^*,
\]
and
\begin{align*}
\ip{\nabla f(u_k)}{u_k - u^*} &= \ip{Q(u_k-u^*)}{u_k - u^*} + \ip{Qu^* + b}{u_k-u^*} \\ 
&= \ip{QA(x_k-z)}{A(x_k - z)} + \ip{QAz}{A(x_k-z)} + \ip{b}{A(x_k-z)} \\ 
& = \|Q^{1/2}A(x_k-z)\|^2 + \ip{g}{Q^{1/2}A(x_k-z)} + b\transp A(x_k-z)\\
& = \left\|\bar A(x_k-z)\right\|^2_g.
\end{align*}
Thus
\[
\ip{\nabla f(u_k)}{u_k - u^*}  \ge \left\|\bar A(x_k-z)\right\|_g \cdot \sqrt{f(u_k) - f^*}
\]
and so Premise 1' follows.

\medskip

Premise 2' is an immediate consequence of \eqref{eq.step.length} and the form of $f$.

\medskip

The first inequality in Premise 3' follows from the choice of $v$.  To prove the second inequality in Premise 3', take $\bar d:=\frac{\bar A(x_k-z)}{\|\bar A(x_k-z)\|_g}$ and 
$$
p := \frac{\|\bar A(x_k-z)\|_g}{\ip{\nabla f(u_k)}{u_k-u^*}}\cdot\matr{Q^{1/2}u_k\\1} = \frac{\left\|\bar A(x_k-z)\right\|_g}{\ip{\matr{Q^{1/2}u_k\\1}}{
 \bar A(x_k-z)}}\cdot\matr{Q^{1/2}u_k\\1}.
$$
This choice of $p$ guarantees that $\ip{p}{\bar d} = 1$.  Furthermore, observe that for $i,j \in \{1,\dots,n\}$
\[
\ip{\matr{Q^{1/2}u_k \\ 1}}{\bar a_i - \bar a_j} = \ip{Qu_k}{a_i-a_j} + \ip{b}{a_i-a_j} = \ip{\nabla f(u_k)}{a_i-a_j}.
\]
Hence the above choice of $p$ yields
\[
\dmax_{\ell\in I(x_k),j\in\{1,\dots,n\}} \ip{p}{\bar a_\ell - \bar a_j} = 
\dfrac{\|\bar A(x_k-z)\|_g}
{\ip{\nabla f(u_k)}{u_k-u^*}} \cdot 
\dmax_{\ell\in I(x_k),j\in\{1,\dots,n\}} \ip{\nabla f (u_k)}{a_\ell - a_j}.
\]
On the other hand, the construction of $\bar \Phi_g(\bar A,x,z)$ yields
\[
\dmax_{\ell\in I(x_k),j\in\{1,\dots,n\}} \ip{p}{\bar a_\ell - \bar a_j} \ge \bar\Phi_g(\bar A,x_k,z)\ge \bar\Phi_g(\bar A).
\]
Putting the above two together we get
\[
\dmax_{\ell\in I(x_k),j\in\{1,\dots,n\}} \ip{\nabla f (u_k)}{a_\ell - a_j}\ge \bar\Phi_g(\bar A) \cdot \dfrac{\|\bar A(x_k-z)\|_g}
{\ip{\nabla f(u_k)}{u_k-u^*}}.
\]
\qed

\end{proof}
}

\blue{
As Example~\ref{ex.def.r} in Section~\ref{sec.examples} shows, the minimum in the expression for $r$ in Theorem~\ref{thm.quad.conv} cannot be omitted because $\frac{\bar \Phi_{g}\left(\bar A\right)^2}{8\diam(Q^{1/2}A)^2} > \frac{1}{2}$ can occur.  
}

\medskip

\blue{The next proposition establishes a property of $\bar \Phi_g$ similar to that stated in Proposition~\ref{prop.restricted.width} for $\Phi$. It also provides a bound on $\bar \Phi_g$ in terms of $\Phi$. 
Proposition~\ref{prop.bar.phi} relies on one more piece of notation.  For $ \bar A \in \R^{(m+1)\times n}$ and $g \in \R^m$ define
\[
\delta(g) = \max_{x\in \Delta_{n-1},z\in Z(g)} \ip{\matr{g\\1}}{\bar A(x-z)} = \max_{x,z\in \Delta_{n-1}} \ip{\matr{g\\1}}{\bar A(x-z)}.\]

\begin{proposition}\label{prop.bar.phi}
Assume $\bar A\in \R^{(m+1)\times n}$ and $g\in \R^m$.  Assume $x\in\Delta_{n-1}$ and $z\in Z(g)$ are such that $\bar A(x-z) \ne 0$, and let $\bar d:=\frac{\bar A(x-z)}{\|\bar A(x-z)\|_g}$.
\begin{description}
\item[(a)] The following identity holds 
\begin{multline}\label{prop.identity.bar.phi}
\bar \Phi_g\left(\bar A,x,z\right) = \\
\max\left\{\lambda: \exists w,y\in \Delta_{n-1}, \; I(w) \subseteq I(x),\; \bar A(w-y) = \lambda \bar d\right\}.
\end{multline}
Furthermore, if $p$ attains the minimum in 
\eqref{eq.def.bar.phi} and $\bar u = \bar Aw, \; \bar v = \bar Ay$ maximize the right hand side in \eqref{prop.identity.bar.phi} then $\bar v\in \conv(\bar B),$ for $\bar B = \Argmin_{\bar a \in \bar A} \ip{p}{\bar a}$, and  $\bar u\in\conv(\bar A \setminus \bar B).$  Furthermore,
\[
\bar \Phi_g(\bar A,x,z) = \|\bar u-\bar v\|_g.
\]
\item[(b)] 
If $\delta(g) = 0$ then 
$
\bar \Phi_g\left(\bar A,x,z\right) = \Phi(A,x,z)
$ for $A = \matr{I_m & 0 }\bar A.$ 

Otherwise
$
\bar \Phi_g\left(\bar A,x,z\right) \ge \Phi(\hat A,x,z)$ for 
$
\hat A = \matr{I_m & 0\\ \frac{g\transp}{\sqrt{\delta(g)}}& \frac{1}{\sqrt{\delta(g)}}}\bar A.
$

In particular, if at least two columns of $\bar A$ are different then $\bar \Phi_g\left(\bar A\right) > 0$ for all $g \in \R^m$. 
\end{description}
\end{proposition}
}
\blue{
\begin{proof}
\begin{description}
\item[(a)] The proof is a straightforward modification of the proof of Proposition~\ref{prop.restricted.width} with $\bar A$ in place of $A,$ $\bar d$ in place of $d$, and $\|\cdot\|_g$ in place of $\|\cdot\|$.
\item[(b)] 
Assume $w,y\in \Delta_{n-1}$ and $\lambda \in \R$.  Observe that $\bar A (w-y) = \lambda d$ if and only if $\bar A (w-y) = t \bar A(x-z)$ for $t = \frac{\lambda}{\|\bar A(x-z)\|_g}$.  Furthermore, for $r\ne 0$ 
\begin{equation}\label{eq.red.eqn}
\bar A(w-y) = t \bar A(x-z)  \Leftrightarrow 
\matr{I_m & 0 \\ rg\transp & r} \bar A(w-y) = t \matr{I_m & 0 \\ rg\transp & r} \bar A(x-z).
\end{equation}

Now consider two possible cases separately.  

\medskip

{\bf Case 1: $\delta(g) = 0$.}  In this case $\ip{\matr{g\\1}}{\bar A y}=\matr{g\transp & 1} \bar Ay$ is the same  for all $y\in \Delta_{n-1}$ and consequently the last row of equations in \eqref{eq.red.eqn}  is redundant.  Thus for $A = \matr{I_m & 0}\bar A$
\[
 \bar A(w-y) = t \bar A(x-z)  \Leftrightarrow A(w-y) = t A(x-z)
\]
From part (a) and Proposition~\ref{prop.restricted.width} we get
\begin{align*}
& \frac{\bar \Phi_g\left(\bar A,x,z\right)}{\|\bar A(x-z)\|_g}  \\
 &= 
\max\left\{t: \exists w,y\in \Delta_{n-1}, \; I(w) \subseteq I(x),\; \bar A(w-y) = t \bar A(x-z) \right\}\\
& = \max\left\{t: \exists w,y\in \Delta_{n-1}, \; I(w) \subseteq I(x),\;  A(w-y) = t A(z-x)\right\}\\
& = \frac{\Phi(A,x,z)}{\|A(x-z)\|}.
\end{align*}
Furthermore, since $\delta(g) = 0$ we have $ \|\bar A(x-z)\|_g = \|A(x-z)\|$ and consequently $\bar \Phi_g\left(\bar A,x,z\right) = \Phi(A,x,z).$

\bigskip

{\bf Case 2: $\delta(g) > 0$.} Let $
\hat A := \matr{I_m & 0 \\ \frac{g\transp}{\sqrt{\delta(g)}} & \frac{1}{\sqrt{\delta(g)}}} \bar A
$. Plugging $r = \frac{1}{\sqrt{\delta}}$ in \eqref{eq.red.eqn} we get 
\[
\bar A(w-u) = t \bar A(x-z) \Leftrightarrow \hat A(w-u) = t \hat A(x-z).
\]
From part (a) and Proposition~\ref{prop.restricted.width} we get
\begin{align*}
& \frac{\bar \Phi_g\left(\bar A,x,z\right)}{\|\bar A(x-z)\|_g} \\
 &= 
\max\left\{t: \exists w,y\in \Delta_{n-1}, \; I(w) \subseteq I(x),\; \bar A(w-y) = t \bar A(x-z) \right\}\\
& = \max\left\{t: \exists w,y\in \Delta_{n-1}, \; I(w) \subseteq I(x),\;  \hat A(w-y) = t \hat A(x-z)\right\}\\
& = \frac{\Phi(\hat A,x,z)}{\| \hat A(x-z) \|}.
\end{align*}
Since $0\le \matr{g\transp & 1}\bar A(x-z) \le \delta(g),$ it follows that
\begin{align*}
\|\hat A(x-z) \|^2 &= 
\|A(x-z)\|^2 + \frac{\left(\matr{g\transp & 1}\bar A(x-z)\right)^2}{\delta(g)} \\ &\le 
\|A(x-z)\|^2 + \vert\matr{g\transp & 1}\bar A(x-z) \vert\\ &= \|\bar A(x-z)\|_g^2
\end{align*}
and so
\[
\bar \Phi_g(\bar A,x,z) = \Phi(\hat A,x,z) \cdot \frac{\|\bar A(z-x)\|_g}{\| \hat A(x-z) \|}\ge 
\Phi(\hat A,x,z).
\]
Finally, observe that if $\delta(g) = 0$ and two columns of $\bar A$ are different, then at least two columns of $A$ as different as well.  Thus
\[
\bar \Phi_g(\bar A) \ge \Phi(A,Z(g)) \ge \Phi(A) > 0.
\]
On the other hand, if $\delta(g) > 0$ and two columns of $\bar A$ are different, then at least two columns of $\hat A$ as different as well.  Thus
\[
\bar \Phi_g(\bar A) \ge \Phi(\hat A,Z(g)) \ge \Phi(\hat A) > 0.
\]
In either case we have $\bar \Phi_g(\bar A) > 0.$
\qed
\end{description}
\end{proof}
}

\section{Composite convex objective}
\label{sec.non-convex}

\blue{We next extend} the main ideas from Section~\ref{sec.quadratic} to the case when the objective is a composite function of the form \[
f(u) = \blue{h}(Eu)+ \ip{b}{u},\]
where $E\in \R^{p\times m}, \; b\in \R^m,$ and $\blue{h}:\R^p \rightarrow \R$ is a strongly convex function with Lipschitz gradient.  

Consider problem~\eqref{min.problem} for this objective function.  If $L$ is an upper bound on the Lipschitz constant of $\nabla \blue{h}$, then 
\[
f(u_k+\blue{\gamma} v) \le f(u_k) + \blue{\gamma}\ip{\nabla f(u_k)}{v} + \frac{L\blue{\gamma}^2 \|Ev\|^2}{2}.
\]
Consequently, $\blue{\gamma}_k$ in Step 9 of Algorithm~\ref{alg:frankwolfe} could be computed as follows
\begin{align}\label{eq.step.length.non}
\blue{\gamma}_k &:= \argmin_{\blue{\gamma} \in [0,\blue{\gamma}_{\max}]} \left\{
\ip{\blue{\gamma}\nabla f(u_k)}{v} + \frac{L\blue{\gamma}^2 \|Ev\|^2}{2}\right\} \notag\\
& = \left\{ 
\begin{array}{ll} 
\min\left\{\blue{\gamma}_{\max}, - \frac{\ip{\nabla f(u_k)}{v}}{L\|Ev\|^2} \right\} & \text{ if } \; Ev\ne 0\\ \blue{\gamma}_{\max} & \text{ if } \; Ev=0. 
\end{array} 
\right.
\end{align}
Once again, linear convergence can be obtained from Theorem~\ref{thm.fw.linear} and the error bound~\cite[Lemma 2.5]{BeckS15}.  However, the more explicit rate of convergence in Theorem~\ref{thm.non-strong} holds.  
 
\begin{theorem}\label{thm.non-strong} Assume in problem~\eqref{min.problem} the objective is $f(u) = \blue{h}(Eu)+ \ip{b}{u}$ where $\blue{h}$ is $\mu$-strongly convex and $\nabla h$ is Lipschitz. 
\blue{Assume the matrix $\bar A = \matr{EA\\ \frac{1}{\mu}b\transp A}$  has at least two different columns as otherwise problem~\eqref{min.problem} is trivial.

If $x_0\in\Delta_{n-1}$ in Step 1 of Algorithm~\ref{alg:frankwolfe} is a vertex of $\Delta_{n-1}$ and the steplength $\blue{\gamma}_k$ is computed as in~\eqref{eq.step.length.non} for some upper bound $L$ on the Lipschitz constant of $\nabla \blue{h}$ then the convergence rate \eqref{lin.conv} holds with
\[
r = \min\left\{\blue{\frac{\mu}{L}\cdot\frac{\bar \Phi_g\left(\bar A\right)^2}{8\diam(EA)^2}},\frac{1}{2}\right\},
\]
where $g = \frac{1}{\mu}\nabla h(EAz)$ for $z \in Z^*$, which does not depend on the specific choice of $z\in Z^*.$ }
\end{theorem}
\blue{
\begin{proof}
This proof is a straightforward modification of the proof of Theorem~\ref{thm.quad.conv}.  It is a consequence of 
$Z^*\subseteq Z(g)$ and the following three premises that hold provided $u_k \ne u^*$.

\bigskip

\noindent
{\bf Premise 1'':} 
\[
\ip{\nabla f(u_k)}{u_k-u^*} \ge \|\bar A(x_k-z)\|_g\sqrt{\mu(f(u_k) - f^*)}.
\]
\bigskip

\noindent
{\bf Premise 2'':} If $\blue{\gamma}_k =-\frac{\ip{\nabla f(u_k)}{v}}{L\|E v\|^2} < \blue{\gamma}_{\max}$ then
\[
f(u_k)-f(u_{k+1}) \le \frac{\ip{\nabla f(u_k)}{v}^2}{2L\|Ev\|^2}.
\]

\bigskip

\noindent
{\bf Premise 3'':} 
\[
-\ip{\nabla f (u_k)}{v} \ge
\frac{\dmax_{\ell\in I(x_k),j\in\{1,\dots,n\}} \ip{\nabla f (u_k)}{a_\ell - a_j}}{
2} \ge \frac{\bar\Phi_g(\bar A)}{2}\frac{\ip{\nabla f (u_k)}{u_k-u^*}}{\|A(x_k-z)\|_g}.
\]

We first show that   $g=\frac{1}{\mu}\nabla h(EAz)$ does not depend on the specific choice of $z \in Z^*$ and $Z^*\subseteq Z(g)$. 

Let $z,z'\in Z^*$.  From the optimality of $u^*=Az$ and $u' = Az'$ we get both
$
0 \le \ip{\nabla f(u^*)}{u'-u^*}$ and $0 \le \ip{\nabla f(u')}{u^*-u'}
$.  Since $\nabla f(y) = E\transp \nabla h(EAz) + b$ and $h$ is $\mu$-strongly convex we get
\begin{align*}
0 &\le -\ip{\nabla f(u^*)-\nabla f(u')}{u^*-u'}\\ &=
-\ip{\nabla h(EAz) - \nabla h(EAz')}{EA(z-z')} \\&\le - \mu \|EA(z-z')\|^2. 
\end{align*}
This implies that $EAz = EAz'$ and thus $g = \frac{1}{\mu}\nabla h(EAz)$ does not depend on the specific choice of $z\in Z^*.$

We next show $Z^*\subseteq Z(g).$  To that end, let $z\in Z^*$ and $u^* = Az$.  The optimality of $u^*$ implies that for all $y = Ax\in \conv(A)$
\begin{align*}
0 &\le \frac{1}{\mu}\ip{\nabla f(u^*)}{u-u^*} \\ &= \frac{1}{\mu}\ip{E\transp \nabla h(Eu^*) + b}{u-u^*} \\&= 
\ip{g}{EAx} + \ip{\frac{1}{\mu}b}{Ax} - \left(\ip{g}{EAz} + \ip{\frac{1}{\mu}b}{Az}\right).
\end{align*}
Hence $z\in  Z(g)$.  Since this holds for all  $z\in Z^*$, we have $Z^* \subseteq Z(g)$.

We next show each of the above three premises. Since $f$ is convex we readily have
\[
\ip{\nabla f(u_k)}{u_k - u^*}  \ge f(u_k) - f^*.
\]
On the other hand, since $h$ is $\mu$-strongly convex we also have
\begin{align*}
\ip{\nabla f(u_k)}{u_k - u^*}  &= \ip{\nabla f(u_k) - \nabla f(u^*)}{u_k - u^*}+\ip{\nabla f(u^*)}{u_k - u^*} \\ & 
= \ip{\nabla h(u_k) - \nabla h(u^*)}{E(u_k - u^*)} +  \ip{\nabla f(u^*)}{u_k - u^*}\\
& \ge \mu\|E(u_k-u^*)\|^2 + \ip{\nabla h(Eu^*)}{E(u_k - u^*)}+ \ip{b}{u_k-u^*}  \\
& = \mu \left(\|EA(x_k-z)\|^2  + \ip{g}{EA(x_k-z)}+ \frac{1}{\mu} b\transp A(x_k-z)\right)\\
&= \mu \|\bar A(x_k - z)\|^2_g.
\end{align*}
Therefore
\[
\ip{\nabla f(u_k)}{u_k - u^*}  \ge   \|\bar A(x_k - z)\|_g \cdot\sqrt{\mu (f(u_k) - f^*)}
\]
and Premise 1'' follows.

\medskip

Premise 2'' is an immediate consequence of \eqref{eq.step.length.non} and the fact the $L$ is a bound on the Lipschitz constant  of $\nabla \blue{h}$.

\medskip

The proof of Premise 3'' is also similar to that of Premise 3' in the  proof of Theorem~\ref{thm.quad.conv}.  The first inequality follows from the choice of $v$.  For the second inequality take $\bar d:=
\frac{A(x_k-z)}{\|A(x_k-z)\|_g}$ and
\begin{align*}
p :&= \frac{\|\bar A(x_k-z)\|_g}{\ip{\nabla f(u_k)}{u_k-u^*}}\cdot\matr{\nabla h(Eu_k)\\ \mu} 
= \frac{\left\|\bar A(x_k-z)\right\|_g}{\ip{\matr{\nabla h(Eu_k)\\ \mu}}{
 \bar A(x_k-z)}}\cdot\matr{\nabla h(Eu_k)\\\mu}.
\end{align*}
This choice of $p$ guarantees that $\ip{p}{\bar d} = 1$.  
Furthermore, observe that for all $i,j\in\{1,\dots,n\}$
\[
\ip{\matr{\nabla h(Eu_k)\\ \mu}}{\bar a_i - \bar a_j} 
 = \ip{E\transp \nabla h(Eu_k)}{a_i - a_j} + \ip{b}{a_i-a_j} 
 = \ip{\nabla f(u_k)}{a_i-a_j}.
\]
Hence the above choice of $p$ yields
\[
\dmax_{\ell\in I(x_k),j\in\{1,\dots,n\}} \ip{p}{\bar a_\ell - \bar a_j} = 
\dfrac{\|\bar A(x_k-z)\|_g}{\ip{\nabla f(u_k)}{u_k-u^*}}\cdot\dmax_{\ell\in I(x_k),j\in\{1,\dots,n\}} \ip{\nabla f (u_k)}{a_\ell - a_j}.
\]
On the other hand, the construction of $\bar \Phi_g(\bar A,x,z)$ yields
\[
\dmax_{\ell\in I(x_k),j\in\{1,\dots,n\}} \ip{p}{\bar a_\ell - \bar a_j} \ge \bar\Phi_g(\bar A,x_k,z)\ge \bar\Phi_g(\bar A).
\]
Putting the above two together we get
\[
\dmax_{\ell\in I(x_k),j\in\{1,\dots,n\}} \ip{\nabla f (u_k)}{a_\ell - a_j}
\ge\bar\Phi_g(\bar A) \cdot \dfrac{
\ip{\nabla f(u_k)}{u_k-u^*}}{\|\bar A(x_k-z)\|_g}.
\]
\qed
\end{proof}
}
\blue{
\section{Some examples}
\label{sec.examples}
Example~\ref{ex.phi.tight} and Example~\ref{ex.bar.phi.tight} below show that the convergence bounds for the Frank-Wolfe algorithm with away steps stated in Theorem~\ref{thm.fw.linear} and Theorem~\ref{thm.quad.conv} in terms of $\Phi$ and $\bar \Phi_g$ are tight modulo some constants. Example~\ref{ex.bar.phi.tight.again} shows that the bound on $\bar \Phi_g$ in terms of $\Phi$ in Proposition~\ref{prop.bar.phi} is also tight modulo a constant.  Finally, Example~\ref{ex.def.r} shows that the minimum in the expressions for $r$ in Theorem~\ref{thm.quad.conv} and Theorem~\ref{thm.non-strong} cannot be omitted.

We illustrate the bounds in Example~\ref{ex.phi.tight} and Example~\ref{ex.bar.phi.tight}  via computational experiments.  Our experiments were conducted via a verbatim implementation of Algorithm~\ref{alg:frankwolfe} in matlab for a convex quadratic objective with steplength computed as in~\eqref{eq.step.length}.
The matlab code is publicly available at the following website
{\tt http://www.andrew.cmu.edu/user/jfp/fwa.html}.  The reader can easily replicate the numerical results described below.

\begin{example}\label{ex.phi.tight} Let $\theta \in (0,\pi/6)$ and 
$
A := \matr{\cos(2\theta)  &1 & -1 \\ \sin(2\theta)&0 & 0}.
$
In this case $\diam(A) = 2$ and $\Phi(A) = \sin(\theta)$.
Consider the problem
\[
\min_{u\in \conv(A)} \frac{1}{2}\|u\|^2. 
\]
The optimal value of this problem is zero attained at $u^* = 0$. 
Furthermore, the condition number of the objective function is one.  If we apply Algorithm~\ref{alg:frankwolfe} to this problem starting with $u_0 = a_1 = Ae_1$, then it follows that for $k=1,2,\dots$ the algorithm alternates between regular steps toward $a_2$ and away steps from $a_1$.  Furthermore, it can be shown via a geometric reasoning that  for $k=1,2,\dots$
\[
\frac{1}{2}\|u_{k+1}\|^2 =  \frac{1}{2}\|u_k\|^2\cos^2(\theta_k)
\]
where $\theta_k \in(0,3\theta)$.  In particular,  for $k=1,2,\dots$
\[
\frac{\frac{1}{2}\|u_{k+1}\|^2 }{\frac{1}{2}\|u_{k}\|^2} = 1 - \sin^2(\theta_k) \ge 1-9\sin^2(\theta) = 1-\frac{36\Phi(A)^2}{\diam(A)^2}.
\]
Thus the rate of convergence \eqref{lin.conv} in Theorem~\ref{thm.fw.linear} is tight modulo a constant.

Figure~\ref{fig.example.strong} shows the ratio $1-\frac{\frac{1}{2}\|u_{k+1}\|^2 }{\frac{1}{2}\|u_{k}\|^2}$ and the bound $9\sin^2(\theta)$ based on numerical experiments for $\theta = \pi/10, \; \theta= \pi/100,$ and $\theta= \pi/1000$.   The figure confirms that the above ratio stays bounded away from zero, that is, the objective values $\frac{1}{2}\|u_{k}\|^2$ converge linearly to zero.  The figure also confirms that in each case the above ratio stays below and  pretty close to the bound $9\sin^2(\theta)$ and thus the rate of linear convergence of $\frac{1}{2}\|u_{k}\|^2$ to zero is slower than $1-9\sin^2(\theta)$.

\begin{figure}[h]
\centering
\includegraphics[width=.485\textwidth]{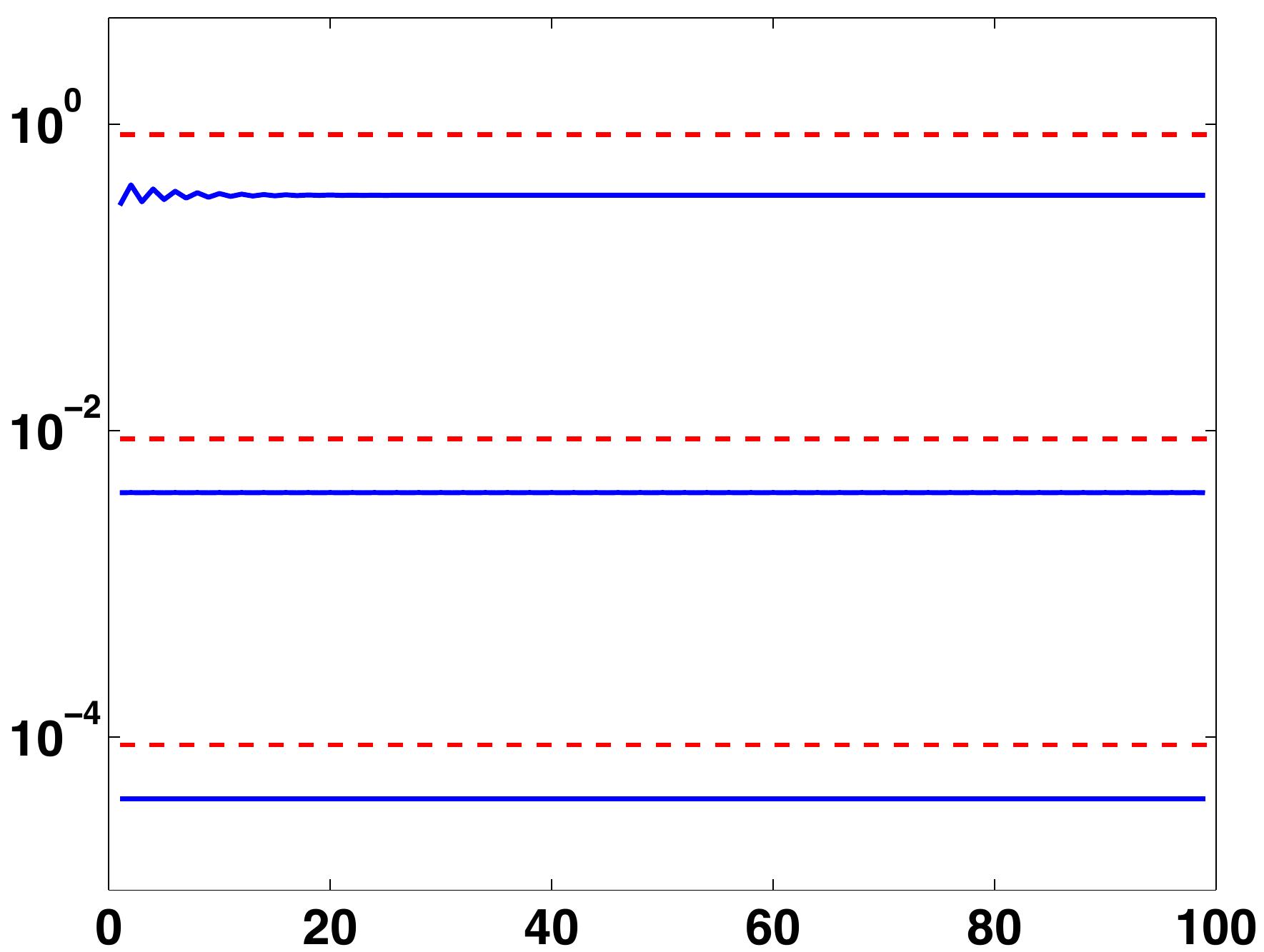}
\caption{\blue{Plot of the ratio $1-\frac{\frac{1}{2}\|u_{k+1}\|^2 }{\frac{1}{2}\|u_{k}\|^2}$ (solid line) and the bound $9\sin^2(\theta)$ (dash line) in Example~\ref{ex.phi.tight} for $\theta = \pi/10$ (top lines), $\theta= \pi/100$ (middle lines), and $\theta= \pi/1000$ (bottom lines).}}
\label{fig.example.strong}
\end{figure}

\end{example}

\begin{example}\label{ex.bar.phi.tight} 
Let $t > 0$ and $A:=\matr{t & t & -t  \\ t & 0 & 0}.$ Consider the problem
\[
\min_{u\in \conv(A)} \frac{1}{2}\ip{u}{Qu} + \ip{b}{u} 
\]
where $Q = \matr{1 & 0 \\ 0 & 0}$ and $b = \matr{0\\1}$.
The optimal value of this problem is zero attained at $u^* = \matr{0\\0}$.
Thus $\bar A = \matr{Q^{1/2}A \\ b\transp A} = \matr{t& t & -t  \\ 0 & 0 & 0 \\ t& 0 & 0}, \; Z^* = \{0.5e_2+0.5e_3\}$, and $g =\matr{0\\ 0}$.  Proposition~\ref{prop.bar.phi}(a) and some straightforward calculations show that $Z(g) = \conv\{e_2,e_3\}$, $\diam(Q^{1/2}A) = 2t,$ and   
$$\bar \Phi_g(\bar A) = \left\{\begin{array}{ll} 2t & \text{ if } \; t < 1/8 \\ 
\sqrt{t-1/16} & \text{ if } \; t \ge 1/8. 
\end{array} \right.$$
 If we apply Algorithm~\ref{alg:frankwolfe} to this problem starting with $u_0 = a_1 = Ae_1$, then it follows that for $k=1,2,\dots$ the algorithm alternates between regular steps toward $a_2$ and away steps from $a_1$.  Furthermore, for $t \gg 1$ it can be shown via a geometric reasoning that for $1\le k < t/4$ 
\[
\frac{\frac{1}{2}\ip{u_{k+1}}{Qu_{k+1}} + \ip{b}{u_{k+1}}}{\frac{1}{2}\ip{u_k}{Qu_k} + \ip{b}{u_k}} \ge 1 - \frac{4}{t}. 
\]
Observe that for $t \gg 1$ we have 
\[
\frac{\bar\Phi_g(\bar A)^2}{8\diam(Q^{1/2}A)^2} = \frac{t-1/16}{32t^2} \approx \frac{1}{32t}.
\]
Therefore, the rate of convergence in Theorem~\ref{thm.quad.conv} is tight modulo a  constant.  Notice that in sharp contrast to $\frac{\bar \Phi_g(\bar A)}{\diam(Q^{1/2}A)}$ 
which tends to zero as $t\rightarrow \infty$, all of $\frac{\Phi(A)}{\diam(A)}, \frac{\Phi(Q^{1/2}A)}{\diam(Q^{1/2}A)},$ and $\frac{\Phi(\bar A)}{\diam(Q^{1/2}A)}$ stay constant and bounded away from zero for all $t > 0$.  Thus the convergence rate in Theorem~\ref{thm.quad.conv} cannot be stated solely in terms of any of the latter three ratios.

Figure~\ref{fig.example.nostrong} shows the ratio $1-\frac{\frac{1}{2}\ip{u_{k+1}}{Qu_{k+1}} + \ip{b}{u_{k+1}}}{\frac{1}{2}\ip{u_k}{Qu_k} + \ip{b}{u_k}}$ and the bound $\frac{4}{t}$ 
based on numerical experiments for $t = 200, \; t= 20000,$ and $t= 2000000$.  Once again, the figure confirms that the objective values converge linearly to zero  and that the ratio stays below and pretty close to the bound $\frac{4}{t}$.  However, we should note that the latter only holds for $k$ up to a certain threshold.  Given the simplicity of this example, the optimal value is attained for $k$ sufficiently large.

\begin{figure}[h]
\centering
\includegraphics[width=.485\textwidth]{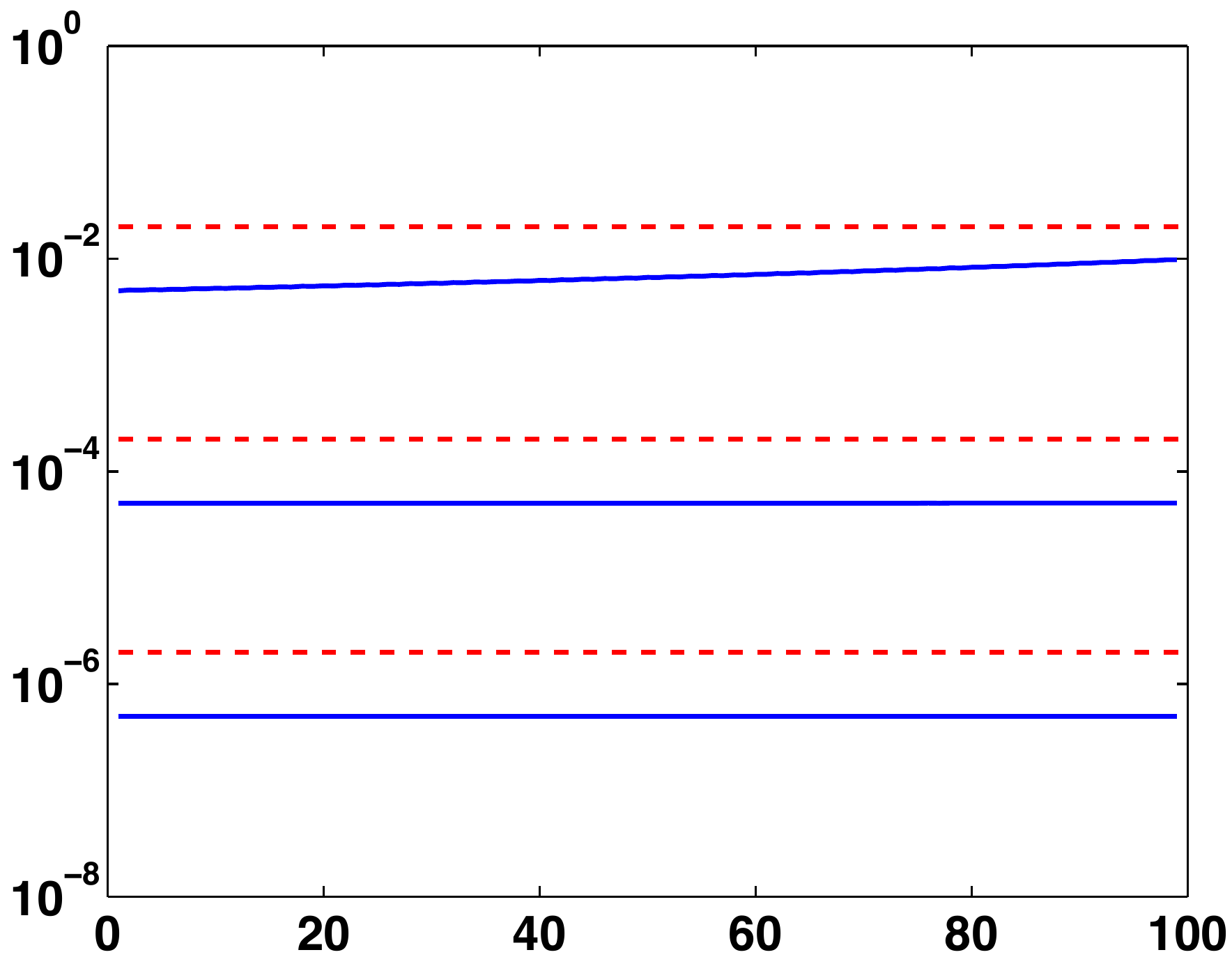}
\caption{\blue{Plot of the ratio $1-\frac{\frac{1}{2}\ip{u_{k+1}}{Qu_{k+1}} + \ip{b}{u_{k+1}}}{\frac{1}{2}\ip{u_k}{Qu_k} + \ip{b}{u_k}}$ (solid line) and the bound $\frac{4}{t}$ (dash line) in Example~\ref{ex.bar.phi.tight} for $t = 200$ (top lines), $t= 20000$ (middle lines), and $t= 2000000$ (bottom lines).}}
\label{fig.example.nostrong}
\end{figure}

\end{example}

\begin{example}\label{ex.bar.phi.tight.again} Let $t > 0$ and $\bar A := \matr{t & t & -t \\ t & 0 & 0}.$ 
For $g = 0$ proceeding as in Example~\ref{ex.bar.phi.tight} it follows that $Z(g) = \conv\{e_2,e_3\}, \delta(g) = t,$ and 
$$\bar \Phi_g(\bar A) = \left\{\begin{array}{ll} 2t & \text{ if } \; t < 1/8 \\ 
\sqrt{t-1/16} & \text{ if } \; t \ge 1/8. 
\end{array} \right.$$
In this case the matrix $\hat A$ in Proposition~\ref{prop.bar.phi}(b) is
\[
\hat A = \matr{1 & 0 \\ 0 & \frac{1}{\sqrt{t}}} \bar A = \matr{t & t & -t \\ \sqrt{t} & 0 & 0}.
\]
Proposition~\ref{prop.restricted.width} and some straightforward calculations show that 
\[
\Phi(\hat A,Z(g)) = \frac{2t}{\sqrt{4t+1}}. 
\]
This shows that the bound in Proposition~\ref{prop.bar.phi}(b) is tight modulo a constant.  Indeed, notice that in this example $\frac{\Phi(\hat A,Z(g))}{\bar\Phi_g(\bar A)} \rightarrow 1$ both when $t\rightarrow \infty$ and when $t\downarrow 0$.
\end{example}

\begin{example}\label{ex.def.r} 
Let $t > 0$ and $A:=\matr{0 & 1 \\ 0 & t}.$ Consider the problem
\[
\min_{u\in \conv(A)} \frac{1}{2}\ip{u}{Qu} + \ip{b}{u} 
\]
where $Q = \matr{1 & 0 \\ 0 & 0}$ and $b = \matr{0\\1}$.
The optimal value of this problem is zero attained at $u^* = \matr{0\\0}$.
Thus $\bar A = \matr{Q^{1/2}A \\ b\transp A} = \matr{0 & 1  \\ 0 & 0 \\ 0 & t}, \; Z^* = \{e_1\}$, and $g =\matr{0\\ 0}$.  Proposition~\ref{prop.bar.phi}(a) and some straightforward calculations show that $Z(g) = \conv\{e_1\}$, $\diam(Q^{1/2}A) = 1,$ and   $\bar \Phi_g(\bar A) = \sqrt{1+t}.$
Thus for $t \gg 1$ 
\[
\frac{\bar \Phi_{g}\left(\bar A\right)^2}{8\diam(Q^{1/2}A)^2} = \frac{1+t}{8} > \frac{1}{2}.
\]
\end{example}

}

\section*{Acknowledgements}

Javier Pe\~na's research has been supported by NSF grant CMMI-1534850.


\end{document}